\documentclass[11pt]{article}
\usepackage[T1]{fontenc}
\usepackage[utf8]{inputenc}
\usepackage[symbol]{footmisc}
\usepackage{amsmath,amssymb,amsfonts,amsthm}
\usepackage{algorithm,algorithmicx,algpseudocode}
\usepackage{fullpage,authblk,mathabx}
\usepackage{hyperref,cite,color}
\usepackage{enumitem}
\usepackage{graphicx}
\usepackage{tikz}
\usepackage{caption}
\usepackage{thm-restate}
\usetikzlibrary{shapes}
\usetikzlibrary{backgrounds}
\usetikzlibrary{positioning}
\usepackage{todonotes}
\usepackage{subcaption}
\usepackage{placeins}

\usetikzlibrary{decorations.pathmorphing}
\usetikzlibrary{decorations.markings,intersections,calc}
\usetikzlibrary{shapes.misc}

\hypersetup{
    colorlinks,
    linkcolor = {red!60!black},
    citecolor = {green!60!black},
    urlcolor  = {blue!60!black},
    pdftitle  = {An Erdos-Posa theorem for induced packings of even cycles},
    pdfauthor = {Ahn, Jaffke, Kwon, Lima, Park}
}
\newtheorem{theorem}{Theorem}[section]
\newtheorem{lemma}[theorem]{Lemma}
\newtheorem{corollary}[theorem]{Corollary}
\newtheorem{proposition}[theorem]{Proposition}

\newtheorem{claim}{Claim}

\newcommand{\tin}{\mathsf{tree} \textnormal{-} \alpha}

\newenvironment{clproof}{\begin{list}{}{%
			\setlength{\leftmargin}{3mm}%
		} \item {\it Proof of Claim.} }{\hfill$\lozenge$\end{list}}

\newcommand\abs[1]{\lvert #1\rvert}
\newcommand{\dist}{\mathrm{dist}}
\newcommand{\internal}{\mathrm{int}}
\newcommand\extrafootertext[1]{%
    \bgroup
    \renewcommand\thefootnote{\fnsymbol{footnote}}%
    \renewcommand\thempfootnote{\fnsymbol{mpfootnote}}%
    \footnotetext[0]{#1}%
    \egroup
}

\newcommand{\ep}{Erd\H{o}s-P\'osa }
\newcommand{\mcf}{\mathcal{F}}
\newcommand{\aux}{\operatorname{Aux}}

\begin{document}
\date{\today}
\title{An Erd\H{o}s-P\'{o}sa theorem for induced packings of even cycles}
\author[1]{Jungho~Ahn}
\author[2]{Lars~Jaffke}
\author[3,4]{O-joung~Kwon}
\author[2]{Paloma~T.~de Lima}
\author[3]{Hyemi~Park}
\affil[1]{Department of Computer Engineering, Inha University, Incheon, South~Korea}
\affil[2]{NHH Norwegian School of Economics, Bergen, Norway}
\affil[3]{Department of Mathematics, Hanyang University, Seoul, South~Korea}
\affil[4]{Discrete Mathematics Group, Institute for Basic Science (IBS), Daejeon, South~Korea}
\affil[ ]{\small\textit{Email addresses:}
    \texttt{junghoahn@inha.ac.kr},
    \texttt{lars.jaffke@nhh.no},
    \texttt{ojoungkwon@hanyang.ac.kr}, 
    \texttt{paloma.lima@nhh.no}
}

\maketitle

\extrafootertext{
O. Kwon and H. Park are supported by the National Research Foundation of Korea (NRF) grant funded by the Ministry of Science and ICT (No. RS-2023-00211670). O. Kwon is supported by the Institute for Basic Science (IBS-R029-C1).}

\begin{abstract}
    We show that for every graph $G$ and every positive integer $k$, $G$ contains either $k$ vertex-disjoint even cycles with no edges between distinct cycles, or a set of $O(k^2\log k)$ vertices whose closed neighborhood meets every even cycle of $G$.
    Moreover, we give a polynomial-time algorithm that outputs one of the two outcomes.
    As an application, we show that every $K_{1,t}$-free graph containing no induced packing of $k$ even cycles has tree-independence number $O(tk^2\log k)$, with a corresponding decomposition computable in polynomial time.
\end{abstract}

\section{Introduction}

A family $\mcf$ of graphs is said to satisfy the \emph{\ep property} if there exists a function~$f$ such that, for every positive integer~$k$, every graph~$G$ contains either~$k$ vertex-disjoint subgraphs isomorphic to graphs in $\mcf$, or a set of at most $f(k)$ vertices intersecting every subgraph of~$G$ isomorphic to a graph in $\mcf$.
Erd\H{o}s and P\'osa~\cite{ErdosP1965} proved in 1965 that the family of cycles satisfies the \ep property, and since then the property has been studied for numerous other families of graphs.
We refer the reader to the survey of Raymond and Thilikos~\cite{RaymondT2017} for an overview of the \ep property.

Among several variants of the \ep property, a \emph{coarse} analogue has recently been proposed and studied.
Georgakopoulos and Papasoglu~\cite{GeorgakopoulosP2025} investigated connections between graph minors and coarse geometry and, among other questions, proposed a coarse analogue of the \ep theorem formulated in terms of fat minors (Conjecture~9.7 in~\cite{GeorgakopoulosP2025}).
Recently, Albrechtsen, Bonamy, Bourneuf, and Davies~\cite{AlbrechtsenBRD2026} confirmed this conjecture.
In contrast, Albrechtsen and Davies~\cite{AlbrechtsenD2025} showed that not every planar graph has the coarse \ep property, even though the original \ep theorem extends to planar~$H$-minor models~\cite{RS1986}.

A set~$\mathcal{F}$ of vertex-disjoint subgraphs of a graph~$G$ is called an \emph{induced packing} in~$G$ if there is no edge of~$G$ between any two subgraphs of~$\mathcal{F}$.
Ahn, Gollin, Huynh, and Kwon~\cite{coarseEP} established a direct induced packing analogue of the \ep theorem by showing that, for every graph~$G$ and every positive integer~$k$, either~$G$ contains~$k$ vertex-disjoint cycles with no edges between distinct cycles, or there is a set of $O(k\log k)$ vertices whose closed neighborhood meets every cycle of~$G$.
Dujmovi\'c, Joret, Micek, and Morin~\cite{DujmovicJMM} studied a more general distance variant of the \ep property.
They showed that, for every positive integer~$d$, either~$G$ contains~$k$ cycles that are pairwise at distance greater than~$d$, or there is a set~$S$ of $O(k^{18}\operatorname{polylog} k)$ vertices such that the ball of radius $19d$ around~$S$ meets every cycle of~$G$.
More recently, Chudnovsky et al.~\cite{ChudnovskyDJGKMMS2026} obtained a result with a hitting set of size $O(k\log k)$ and radius $O(d)$.
Thus, the results of Ahn et al.
and Chudnovsky et al.
provide closely related coarse analogues of the classical \ep theorem, corresponding respectively to induced packings and to packings of cycles that are far apart.
Moreover, both results have been extended to cycles of prescribed minimum length in their papers.

One natural direction for extending these results is to identify other families of graphs that admit an induced packing analogue of the \ep theorem beyond the family of cycles.
Two particularly notable recent developments are results for planar~$H$-minors by Chudnovsky, Reinald, and Thomass\'e~\cite{ChudnovskyRT2026}, and for long holes and theta graphs by Czy\.{z}ewska et al.~\cite{CzyzewskaMPRR2026}.
These results demonstrate that induced \ep phenomena extend to substantially broader settings.

Our main result is an induced packing \ep theorem for even cycles, together with a polynomial-time algorithm that finds the corresponding packing or hitting set.

\begin{theorem}\label{thm:main}
    For every graph~$G$ and every positive integer~$k$, $G$ contains either an induced packing of~$k$ even cycles, or a set of $O(k^2\log k)$ vertices whose closed neighborhood meets every even cycle.
    Moreover, there is an algorithm running in time polynomial in $\abs{V(G)}$ and~$k$ that outputs one of these two outcomes.
\end{theorem}

We point out that the two recent results mentioned above are closely related to our work.
It is well known that every $\theta$-minor contains an even cycle, where the $\theta$ graph consists of two vertices with three parallel edges.
By applying the result of Chudnovsky, Reinald, and Thomass\'e~\cite{ChudnovskyRT2026} to the $\theta$ graph, and then greedily packing even cycles in $\theta$-minor-free graphs by using the cactus structure, one can obtain a qualitative induced packing \ep theorem for even cycles.
However, their result does not provide an explicit bound on the size of the hitting set and is not algorithmic.

The result of Czy\.zewska et al.~\cite{CzyzewskaMPRR2026} establishes an induced packing \ep theorem for $\theta$ induced minors.
Although their theorem does not directly apply to $\theta$-minors, their proof strategy appears to yield an induced packing \ep theorem for even cycles with gap function $O(k\log k)$.
However, this approach is again non-algorithmic. 

In contrast to these results, our proof is algorithmic.
We believe that, with a more refined analysis, the bound could likely be improved to $O(k\log k)$.
However, we chose to present the current argument, as it gives a cleaner and more transparent proof. 

The \emph{tree-independence number} is a tree-decomposition parameter
introduced by Dallard, Milani\v{c}, and \v{S}torgel~\cite{DallardMS2024-2}
that has found several applications to algorithms for independent-set
problems.  As an application of our main theorem, we show that every
$K_{1,t}$-free graph containing no induced packing of $k$ even cycles has
tree-independence number $O(tk^2\log k)$, and construct a tree decomposition
witnessing this bound in polynomial time.

A natural further direction is to consider even cycles that are farther apart. More precisely, one may ask whether there are functions $f$ and $g$ such that every graph contains either $k$ even cycles that are pairwise at distance greater than $d$, or a set $X$ of at most $f(k,d)$ vertices such that the ball of radius $g(d)$ around $X$ intersects every even cycle. Our result can be viewed as the case $d=1$. More broadly, it would also be interesting to generalize the theorem of Chudnovsky, Reinald, and Thomass\'e~\cite{ChudnovskyRT2026} for planar $H$-minors to the far-apart setting; a related question was also raised
by Albrechtsen and Davies~\cite{AlbrechtsenD2025}.

\paragraph*{Sketch of the proof.}

We begin with a standard peeling argument.
Repeatedly, we find a shortest even cycle~$C$, remove its closed neighborhood, and continue in the remaining graph.
If this process finds~$k$ cycles, then these cycles form an induced packing, and we are done.
Thus, we may assume that the process terminates after at most $k-1$ iterations.
It therefore suffices to consider the following restricted setting: $G$ contains a shortest even cycle~$C$ such that $G-N_G[C]$ contains no even cycle, where $N_G[C]$ denotes the closed neighborhood of the vertex set of~$C$.
In this restricted setting, our goal is to find either an induced packing of~$k$ even cycles, or a set~$X$ of $O(k\log k)$ vertices such that $G-N_G[X]$ contains no even cycle.
The statement is written in Theorem~\ref{thm:main2}.

Our proof introduces two structures, called skeletons and path-matching graphs, which may be of independent interest. 

A \emph{skeleton} $H$ in a graph~$G$ is a bipartite subgraph in which every vertex has degree between~$2$ and~$4$, all chords of~$H$ in~$G$ are parity-breaking and form a matching, no chord of~$H$ has both endpoints at branching vertices of~$H$, and no vertex of degree~$4$ in~$H$ is incident with a chord.
Here, a \emph{chord} means an edge in $E(G)\setminus E(H)$ connecting two vertices of~$H$.

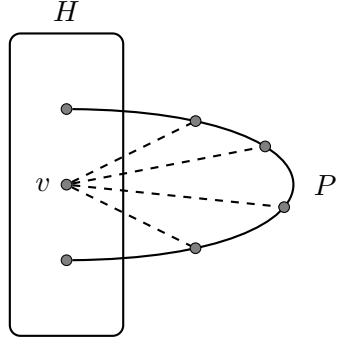
\begin{figure}[t]
    \centering
    \begin{tikzpicture}
    \tikzstyle{R}=[circle, draw, fill=black!50, opacity=0, inner sep=0pt,minimum width=0pt]
    \tikzstyle{v}=[circle,draw,fill=black!50,inner sep=0pt,minimum width=4pt]

    \path[use as bounding box] (-4.5,-0.2) rectangle (1.6,4.8);
    
    \draw(-2.25,4.7) node [label=below:$H$]{};
    \draw[rounded corners, thick] (-2,0) -- (-1.5,0) -- (-1.5,4) --(-3,4)--(-3,0)--(-2,0);
    \draw (-2.25, 3) node [v] (v)[]{};
    \draw (-2.25, 1) node [v] (w) []{};
    \draw (-2.25,2) node [v] (a) [label=left:$v$]{};
    \draw[thick, postaction={decorate}, 
    decoration={markings,
    mark =at position 0.4 with {\coordinate (c) at (0,0);},
    mark =at position 0.25 with {\coordinate (b) at (0,0);},
    mark =at position 0.5 with {\coordinate (P) at (0,0);},
    mark =at position 0.55 with {\coordinate (d) at (0,0);},
    mark =at position 0.75 with {\coordinate (e) at (0,0);}
    }] 
    (v) to[out=-0.5, in=0.5, looseness=5] (w);

    \draw[dashed, thick] (a)--(b);
    \draw[dashed, thick] (a)--(c);
    \draw[dashed, thick] (a)--(d);
    \draw[dashed, thick] (a)--(e);
    
    \draw (b) node[v] []{};
    \draw (c) node[v] []{};
    \draw (d) node[v] []{};
    \draw (e) node[v] []{};
    \draw (P) node[] [label = right: $P$]{};
    \end{tikzpicture}
    \caption{Many chords (dashed edges) may appear when adding a parity-preserving~$F$-path.}
    \label{fig:manychords}
\end{figure}

In the first phase of the proof, we recursively construct bipartite subgraphs of low degree while controlling their chords, with the goal of eventually obtaining a skeleton with many branching vertices.
For a subgraph~$F$ of~$G$, an \emph{$F$-path} in~$G$ is a path of length at least~$1$ whose endpoints are in $V(F)$, whose internal vertices are not in $V(F)$, and which is not an edge of~$F$.
Starting from the given cycle~$C$, we recursively build a bipartite subgraph by attaching a parity-preserving $F$-path to the previously constructed graph~$F$ at each step.
In Section~\ref{sec:extendable}, we analyze how chords can appear when extending such bipartite subgraphs, and we formally introduce this recursive construction in Subsection~\ref{subsec:bipartiteear}.
Note that we cannot simply bound the number of new chords introduced at each step.
For example, a vertex~$v$ in the previously constructed graph~$F$ may form chords with several vertices of a newly added~$F$-path, see Figure~\ref{fig:manychords} for an example.
However, we can show that the subgraph with all such chords admits a vertex cover of constant size; see Proposition~\ref{prop:boundingchord} for details.
The assumption that $G-N_G[C]$ contains no even cycle is crucial for this proposition, as without it such a constant bound need not hold.
We use these vertex covers to simplify the resulting structure so that the remaining chords form a matching. 

In Lemma~\ref{lem:manybranching}, we show that if a graph~$G$ has a skeleton with at least $208\cdot k\log k$ branching vertices, then~$G$ contains an induced packing of~$k$ even cycles.
Therefore, we may assume that the first (maximal) bipartite subgraph we construct, say $H'$, has $O(k\log k)$ branching vertices.
We can also obtain a set~$S'$ with $\abs{S'}=O(k\log k)$ such that $H'-S'$ has no chords and every branching vertices of~$H'$ is contained in~$S'$ (the sets~$H'$ and~$S'$ correspond to~$H_m$ and $\mathcal{S}_m$ respectively, in Subsection~\ref{subsec:bipartiteear}).
By the maximality of the construction, there is no parity-preserving $H'$-path in $G-N_{G-(V(H')\setminus S')}[S']$ that extends~$H'$.
However, parity-breaking paths may still exist, and several of them may combine to form an even cycle.

Path-matching graphs provide a convenient framework for describing the next step of the proof.
A graph~$J$ is a \emph{path-matching} graph if $J=(P_1\cup\cdots\cup P_m) + N$ where $P_1,\dots,P_m$ are pairwise vertex-disjoint paths and~$N$ is a matching on $V(J)$ such that
$N\cap \bigcup_{i\in[m]}E(P_i)=\emptyset$.
We call edges of~$N$ \emph{links}.
In our application, the paths $P_1,\dots,P_m$ correspond to the components of $H'-S'$, while the links of~$N$ represent parity-breaking paths.

\begin{figure}[t]
    \centering
    \begin{tikzpicture}[x=1cm,y=1cm,scale=0.8]
    \tikzset{
        base/.style={thick},
        link/.style={very thick, magenta}
    }

    \draw[base] (0,4) -- (12,4)
        node[right] {$P_1$};
    \draw[base] (0,2) -- (12,2)
        node[right] {$P_2$};
    \draw[base] (0,0) -- (12,0)
        node[right] {$P_3$};

    \draw[link]
        (4.5,4)--(4.6,2)
        node[pos=0.3,left=2pt] {$a$};

    \draw[link]
        (3.6,2)
        to[out=65,in=115,looseness=1.2]
        node[pos=0.7,above right] {$b$}
        (6.0,2);

    \draw[link]
        (7.6,2)
        to[out=60,in=120,looseness=1.2]
        node[pos=0.7,above right] {$c$}
        (10.0,2);

    \draw[link]
        (2.3,2)-- (2.4,0)
        node[pos=0.3,left=2pt] {$d$};

    \draw[link]
        (3.4,0)
        to[out=65,in=115,looseness=1.2]
        node[pos=0.7,above right] {$e$}
        (5.9,0);

    \draw[link]
        (7.2,0)
        to[out=60,in=120,looseness=1.2]
        node[pos=0.7,above right] {$f$}
        (9.8,0);

    \end{tikzpicture}
    \caption{An illustration of serial links in a path-matching graph.
    We say that~$b$, $c$ are serial, and $e$, $f$ are serial.
    In particular, we say that~$e$ and~$f$ are consecutive serial, because the minimal subpath of $P_3$ containing all the endpoints of~$e$ and~$f$ does not meet any other matching edge.
    We construct the packing iteratively so that no two links are consecutive serial.}
    \label{fig:consecutive-serial}
\end{figure}
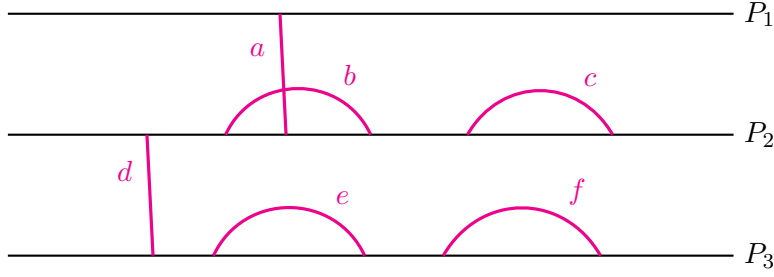

A natural idea is to greedily pack parity-breaking paths and show that, if sufficiently many such paths are found, then one can obtain an induced packing of~$k$ even cycles.
However, this is not true in general.
As illustrated in Figure~\ref{fig:consecutive-serial}, we may pack many parity-breaking paths whose corresponding links are consecutive serial, meaning that they appear consecutively with no other parity-breaking path intersecting the subpath containing their endpoints.
To avoid such configurations, we recursively pack either a path linking two components of $H'-S'$, or a pair of paths that cross or are nested. 

In Proposition~\ref{prop:many-links}, we show that if
$J=(P_1\cup\cdots\cup P_m)+N$
is a path-matching graph in which no two links of~$N$ are consecutive serial, and
$\abs{N}>11m+23g_\theta(k)$,
then~$J$ contains~$k$ pairwise vertex-disjoint cycles, each containing an even number of edges of~$N$.
Here, $g_\theta$ denotes a gap function for the ordinary \ep property of $\theta$-minor models, obtained by Chatzidimitriou, Raymond, Sau, and Thilikos~\cite{ChatRST2018}.
By adapting this result to our setting, we can conclude that the process of adding parity-breaking paths terminates after $O(k\log k)$ steps.

After removing the endpoints of the added parity-breaking paths together with their closed neighborhoods, the remaining graph has a restricted structure.
Let $H''$ be the graph obtained from $H'-S'$ after these deletions.
Then every remaining even cycle meets $H''$ in at most one vertex.
Since we assumed that $G-N_G[C]$ contains no even cycle, every such cycle has an edge to~$H''$ or contains a vertex of~$C$.
By greedily selecting such cycles, we can either obtain an induced packing of~$k$ even cycles, or find a small set whose closed neighborhood meets every remaining even cycle.
This completes the combinatorial part of the proof of Theorem~\ref{thm:main}.

To make the proof algorithmic, we need a polynomial-time algorithm for finding a shortest even cycle, as well as one for finding a shortest $(a,b)$-path of prescribed parity.
A shortest even cycle in a graph can be found in $O(|V(G)|^2)$ time by the algorithm of Yuster and Zwick~\cite{YusterZ1997}.
A shortest $(a,b)$-path of prescribed parity can be found in $O(|V(G)|^2)$ time by the algorithm of Derigs~\cite{Derigs1985}.
Using these algorithms, we can construct both the initial bipartite subgraph and the subsequent parity-breaking paths in polynomial time.

The procedure for finding an induced packing of even cycles from a skeleton is recursive.
We either find a shortest even cycle and recurse on the remaining structure, or, when a branching path meets many chords, first modify the skeleton by the appropriate reduction and then recurse.
This modification is given by $\phi(H,R)$ in Section~\ref{sec:skeleton}.
Both steps can be implemented in polynomial time.
The procedure for finding vertex-disjoint cycles containing an even number of edges of~$N$ in a path-matching graph relies on the algorithm for packing or covering $\theta$-minor models, see Theorem~\ref{thm:thetaminor}.
The remainder of the argument is a simple greedy process.
A detailed description is given in Subsection~\ref{subsec:mainproof}.

\paragraph*{Organization} 

In Section~\ref{sec:prelim}, we introduce basic terminology and notation.
In Section~\ref{sec:extendable}, we analyze how chords can arise when a bipartite subgraph is extended by a parity-preserving path.
In Section~\ref{sec:skeleton}, we introduce skeletons and show that every skeleton with sufficiently many branching vertices contains an induced packing of~$k$ even cycles.
In Section~\ref{sec:pathmatching}, we introduce path-matching graphs and show that every path-matching graph with sufficiently many matching edges, subject to a certain condition, contains~$k$ vertex-disjoint cycles, each containing an even number of matching edges.
Finally, in Section~\ref{sec:reduction}, we introduce the ear decompositions needed for the proof and prove the main theorem.

\section{Preliminaries}\label{sec:prelim}

In this paper, all graphs have neither loops nor multiple edges.
All logarithms are taken to base~$2$.
For a positive integer $n$, we write $[n]=\{1,\ldots,n\}$.

Let~$G$ be a graph.
We write $V(G)$ and $E(G)$ for the vertex set and edge set of~$G$, respectively.
For a set~$S$ of vertices in~$G$, let $G[S]$ denote the subgraph of~$G$ induced by~$S$.
We denote by $G-X$ the graph $G[V(G)\setminus X]$.
For a vertex~$v$ of~$G$, we denote by~$N_G(v)$ the set of neighbors of~$v$ in~$G$ and define~$N_G[v]=N_G(v)\cup\{v\}$.
For a set $X\subseteq V(G)$, let $N_G[X]=\bigcup_{v\in X}N_G[v]$ and let $N_G(X)=N_G[X]\setminus X$.
We say that $N_G[X]$ is the \emph{closed neighborhood} of~$X$ in~$G$.
For a subgraph~$H$ of~$G$, we also write $N_G[H]=N_G[V(H)]$.
For $v\in V(G)$, we write $\deg_G(v)=|N_G(v)|$, and the \emph{average degree} of~$G$ is
$\big(\sum_{v\in V(G)}\deg_G(v)\big)/\abs{V(G)}$.

For two graphs~$G$ and~$H$, their union is $G\cup H=(V(G)\cup V(H), E(G)\cup E(H))$ and their intersection is $G\cap H=(V(G)\cap V(H), E(G)\cap E(H))$.
For a graph~$G$ and a set~$F$ of unordered pairs of distinct vertices of~$G$, let $G+F$ be the graph $(V(G), E(G)\cup F)$.
A graph~$G$ is called \emph{$H$-subgraph-free} if~$G$ contains no subgraph isomorphic to~$H$.

For a path $P=v_1v_2 \cdots v_q$, the vertices~$v_1$ and~$v_q$ are called its \emph{endpoints} and other vertices are called \emph{internal vertices}.
We also say that~$P$ is a $(v_1,v_q)$-path.
Let $\internal(P)$ denote the subgraph of~$P$ induced by the internal vertices of~$P$.
We say that paths~$P$ and~$Q$ are \emph{internally disjoint} if $\internal(P)$ and $\internal(Q)$ are disjoint.
For two vertices~$v$ and~$w$ in~$P$, let $vPw$ denote the subpath of~$P$ whose endpoints are~$v$ and~$w$.
For two vertices~$v$ and~$w$ in~$G$, let $\dist_G(v,w)$ denote the \emph{distance} between~$v$ and~$w$ in~$G$, that is, the length of a shortest $(v,w)$-path of~$G$.
For a subgraph~$H$ of~$G$, an \emph{$H$-path} in~$G$ is a path of length at least~$1$ whose endpoints are in $V(H)$, whose internal vertices are not in $V(H)$, and which is not an edge of~$H$.

For a positive integer $\ell$, let $V_{\geq \ell}(G)$ denote the set of all vertices of degree at least $\ell$ in~$G$, and let $V_{\ell}(G)$ denote the set of all vertices of degree exactly $\ell$ in~$G$.
We say that a graph is a \emph{$(2,4)$-graph} if every vertex of~$G$ has degree~$2$, $3$, or~$4$.
A \emph{corridor} of a $(2,4)$-graph~$G$ is either a cycle of~$G$ containing at most one vertex of $V_{\geq 3}(G)$, or a path in~$G$ of length at least~$1$ whose endpoints are in $V_{\geq 3}(G)$ and internal vertices are not in $V_{\geq 3}(G)$.

A set~$F$ of edges in~$G$ is a \emph{matching} if no two edges in~$F$ share a vertex.
A set~$S$ of vertices in~$G$ is a \emph{vertex cover} if $G-S$ has no edges.
Let~$H$ be an induced subgraph of a graph~$G$.
For a component~$K$ of $G-V(H)$, the \emph{$H$-bridge} of~$G$
corresponding to~$K$ is the subgraph consisting of~$K$, all edges
between~$K$ and~$H$, and their endpoints in~$H$.
The vertices of an~$H$-bridge that belong to~$H$ are called its
\emph{attachments}.

For a subgraph~$H$ of a graph~$G$, an edge in $G[V(H)]$ that is not in~$H$ is called a \emph{chord} of~$H$ in~$G$.
If~$e$ is a chord of~$H$ in~$G$, we also say that~$H$ has a chord~$e$ in~$G$.

Let~$H$ be a bipartite subgraph of~$G$ with bipartition $(A,B)$.
A function $\beta:V(H) \to \{0,1\}$ defined by
        \[\beta(z) = \begin{cases}
               0 & z \in A \\
               1 & z \in B
           \end{cases}\]
is a \emph{bipartition indicator} of~$H$.
An~$H$-path~$P$ in~$G$ is called an \emph{$H$-extendable} path in~$G$ if $H\cup P$ is bipartite. 
We also say that such a path~$P$ is \emph{parity-preserving} for~$H$.
If $H\cup P$ is not bipartite, then we say that~$P$ is \emph{parity-breaking} for~$H$.
Also, a chord of~$H$ in~$G$ is \emph{parity-preserving} for~$H$ if $(V(H), E(H)\cup \{e\})$ is bipartite, and it is \emph{parity-breaking} otherwise.

\begin{lemma}\label{lem:indicatornot}
     Let~$G$ be a graph, and~$H$ be a connected bipartite subgraph of~$G$ with bipartition indicator~$\beta$.
Let~$P$ be an~$H$-path in~$G$ with endpoints~$v$ and~$w$.
Then~$P$ is~$H$-extendable if and only if $\abs{E(P)} \equiv \beta(v) + \beta(w) \pmod{2}$.
\end{lemma}
\begin{proof}
Suppose first that $P$ is $H$-extendable. Then $H\cup P$ is bipartite.
Since $H$ is connected, the restriction to $H$ of any bipartition
indicator of $H\cup P$ is either $\beta$ or $1-\beta$.
Along $P$, the bipartition indicator alternates at every edge.
Therefore,
$|E(P)|\equiv \beta(v)+\beta(w)\pmod 2$.

Conversely, suppose that
$|E(P)|\equiv \beta(v)+\beta(w)\pmod 2$.
Starting from $\beta(v)$, assign alternating values in $\{0,1\}$ to
the internal vertices of $P$ along $P$. The value assigned to $w$ is
$\beta(v)+|E(P)|\equiv \beta(w)\pmod 2$.
Thus, this assignment agrees with $\beta$ at both endpoints of $P$.
Since the internal vertices of $P$ do not belong to $H$, it extends
$\beta$ to a bipartition indicator of $H\cup P$. Hence $P$ is
$H$-extendable.
\end{proof}

We use the following lemma regarding the length of a shortest cycle in graphs of minimum degree at least~$3$.

\begin{lemma}[See Corollary 1.3.5 in~\cite{Diestel2025}]\label{lem:girth}
    If the minimum degree of a graph~$G$ is at least~$3$, then~$G$ has a cycle of length at most $2\log \abs{V(
    G)}$.%
\end{lemma}

\section{Extendable paths and chords}\label{sec:extendable}

When extending a bipartite subgraph by a shortest parity-preserving path, the new chords need not be few. The following proposition shows that they nevertheless have a vertex cover of constant size. Its two conclusions will be used separately: the first produces even cycles at vertices incident with several new chords, while the second bounds the number of vertices that must be set aside at each step of the ear construction.

\begin{proposition}\label{prop:boundingchord}
    Let~$G$ be a $C_4$-subgraph-free graph, and~$C$ be a shortest even cycle in~$G$ such that $G-N_G[C]$ has no even cycle.
Let~$H$ be a connected bipartite subgraph of~$G$ containing~$C$, and $S\subseteq V(H)$.
Let~$P$ be a shortest path among all~$H$-extendable paths contained in $G-N_{G-(V(H)\setminus S)}[S]$, and suppose that~$P$ has length at least~$3$.
Let~$M$ be the set of chords of $H\cup P$ in~$G$ having an endpoint in $\internal(P)$, $J=(V(H\cup P), M)$, and~$Z$ be the set of all vertices of degree~$1$ in~$J$ whose unique neighbor has also degree~$1$ in~$J$. 

    \begin{enumerate}[label=(\arabic*)]
        \item For every $w\in V_{\ge 2}(J)$, there exist two distinct vertices
$z_1,z_2\in V(P)\setminus \{w\}$ such that
$wz_1,wz_2\in M$ and $w\notin V(z_1Pz_2)$. 
    \item The set $V_{\ge 2}(J)\cup Z$ is a vertex cover of~$J$ of size at most $10$.
    \end{enumerate}
\end{proposition}

\subsection{Chords in a shortest even cycle}
We first show that every shortest even cycle of length at least~$6$ in a graph has at most one chord.

 Two chords $vw$ and $xy$ of a cycle~$C$ \emph{cross} if for any path~$P$ from~$v$ to~$w$ in~$C$, $P$ has exactly one vertex of $\{x,y\}$.

\begin{lemma}\label{lem:onechord}
    Let~$C$ be a shortest even cycle in a graph~$G$.
If~$C$ has a chord $vw$, then the length of any path from~$v$ to~$w$ in~$C$ is even.
\end{lemma}
\begin{proof}
    Let $P_1$ and $P_2$ be the two paths from~$v$ to~$w$ in~$C$.
As~$C$ is even, $P_1$ and $P_2$ have the same parity.
Suppose for contradiction that each $P_i$ has odd length.
Since~$G$ has no multiple edge, $P_1$ and $P_2$ have length at least~$3$.
Thus, $P_2\cup vw$ is an even cycle shorter than~$C$, a contradiction. 
\end{proof}

\begin{lemma}\label{lem:shortestevencycle-chords}
    Let~$C$ be a shortest even cycle in a graph~$G$  that has length at least~$6$.
Then~$C$ has at most one chord.
\end{lemma}
\begin{proof}
    Suppose that there are two chords $e=xy$ and $f=vw$ of~$C$.
    We distinguish the case depending on whether they cross or not. %
    
    First, assume that~$e$ and~$f$ do not \emph{cross}.
Without loss of generality, we may assume that $y \neq w$, and that every path from~$y$ to~$w$ in~$C$ contains either both~$x$ and~$v$, or neither of them as internal vertices. 
    Let $Q_1, Q_2, Q_3, Q_4$ be the subpaths of~$C$ such that 
    \begin{itemize}
        \item the sets of endpoints of $Q_1, Q_2, Q_3, Q_4$ are $\{x,v\}, \{v,w\}, \{w,y\}, \{y,x\}$, respectively, and
        \item the internal vertices of $Q_i$ do not contain vertices of $\{x,v,y,w\}$.
    \end{itemize}

    By Lemma~\ref{lem:onechord}, we obtain $
        |E(Q_4)|\equiv 0\pmod{2}$ and $\abs{E(Q_2)}\equiv  0\pmod{2}$.
    Then $\abs{E(Q_1)}+|E(Q_3)|\equiv 0\pmod{2}$.
It follows that $Q_1\cup vw \cup Q_3 \cup yx$ is an even cycle.
    Since $\abs{\{xy, vw\}}<4\le |E(Q_4)|+|E(Q_2)|$, this cycle is shorter than~$C$, contradicting the assumption that~$C$ is shortest.

    Next, we consider the case when~$e$ and~$f$ cross. 
    Let $R_1, R_2, R_3, R_4$ be the subpaths of~$C$ such that 
    the sets of endpoints of $R_1, R_2, R_3, R_4$ are $\{x,v\}, \{v,y\}, \{y,w\}, \{w,x\}$, respectively, and
    the internal vertices of $R_i$ do not contain vertices of $\{x,v,y,w\}$.
    Since the length of~$C$ is at least~$6$, one of $R_1, R_2, R_3, R_4$, say~$R$, has length greater than~$1$. 
    Let $C_1=R_1\cup xy\cup R_3\cup wv$ and $C_2=R_4\cup wv\cup R_2\cup yx$.
    
    By Lemma~\ref{lem:onechord}, we have
    $\abs{E(R_3)} + |E(R_4)|\equiv |E(R_1)| + |E(R_2)|\equiv 0\pmod{2}$ and
    $\abs{E(R_4)} + |E(R_1)|\equiv |E(R_2)| + |E(R_3)|\equiv 0\pmod{2}$.
    Thus, we get 
    \begin{align*}
        |E(C_1)|\equiv |E(R_1)|+|E(R_3)|\equiv  (|E(R_1)| + |E(R_4)|) + (|E(R_4)| + |E(R_3)|) 
        \equiv 0\pmod{2}, \\
        |E(C_2)|\equiv |E(R_4)|+|E(R_2)|\equiv (|E(R_4)| + |E(R_3)|) + (|E(R_2)| + |E(R_3)|) \equiv 0\pmod{2}.
    \end{align*}
    This implies that $C_1$ and $C_2$ are even cycles. 
    Since~$C$ is a shortest even cycle, 
    $\abs{E(C)}\le |E(C_i)|$ for each $i\in [2]$.
Thus, $2|E(C)|\le |E(C_1)|+|E(C_2)|=|E(C)|+4$ and it implies that $\abs{E(C)}\le 4$.
This contradicts the assumption that~$C$ has length at least~$6$.
\end{proof}

    \subsection{Chords whose endpoints are in an~$H$-extendable path}
    In the rest of the section, we consider chords that arise by adding a new~$H$-extendable path.
We first analyze the chords whose endpoints are contained in an~$H$-extendable path.

    \begin{lemma}\label{lem:onechordforpath}
        Let~$G$ be a graph, and let~$H$ be a bipartite subgraph of~$G$ and $S \subseteq V(H)$.
Let~$P$ be a shortest path among all~$H$-extendable paths in $G-N_{G-(V(H)\setminus S)}[S]$.
If~$H\cup P$ has a chord $vw$ with $v,w\in V(P)$, then the subpath $vPw$ has even length.
    \end{lemma}
    \begin{proof}
    Suppose that~$H\cup P$ has a chord $vw$ with $v,w\in V(P)$.
Since~$G$ has no multiple edges, the subpath $vPw$ has length at least~$2$.
If $vPw$ has odd length, then replacing $vPw$ with the edge $vw$ gives an~$H$-extendable path shorter than~$P$, a contradiction.
Thus $vPw$ has even length.
    \end{proof}

    \begin{lemma}\label{lem:shortestHpath-chordsofpath}
    Let~$G$ be a $C_4$-subgraph-free graph and let~$H$ be a connected bipartite subgraph of~$G$ and $S\subseteq V(H)$.
Let~$P$ be a shortest path among~$H$-extendable paths in $G-N_{G-(V(H)\setminus S)}[S]$.
The following hold.
    \begin{enumerate}[label=(\arabic*)]
        \item~$P$ has no two chords $ab$ and $cd$ where $aPb$ and $cPd$ are internally disjoint.
        \item~$P$ has no two chords $ab$ and $cd$ such that one vertex of $\{c,d\}$ is an internal vertex of $aPb$ and the other is not a vertex of $aPb$.
    \end{enumerate}
    \end{lemma}
    \begin{proof}
        (1) Suppose that~$P$ has such chords.
See Figure~\ref{fig:chords2} for an illustration.
By Lemma~\ref{lem:onechordforpath}, we have $\abs{E(aPb)} \equiv |E(cPd)| \equiv 0 \pmod{2}$. 
        Then the path obtained from~$P$ by 
        replacing $aPb$ and $cPd$ with $ab$ and $cd$, respectively, has the same endpoints and the same parity as~$P$.
Thus, this path is a shorter~$H$-extendable path, which is a contradiction.

        \medskip
        (2) Without loss of generality, we may assume $aPb$ contains~$c$ as an internal vertex.
Let~$v$ and~$w$ be the endpoints of~$P$ so that $v,a,c,b,d,w$ appear in this order when we traverse from~$v$ to~$w$ in~$P$.
        
        By Lemma~\ref{lem:onechordforpath}, $\abs{E(aPb)} \equiv |E(cPd)| \equiv 0 \pmod{2}$.
It follows that
        $\abs{E(aPc)}+|E(bPd)| \equiv |E(aPb)| + |E(cPd)| \equiv 0 \pmod{2}$.
        Therefore, replacing $aPc$ and $bPd$ with the edges $ab$ and $cd$ preserves the parity.
        Let 
        \[
        P' = vPa \cup ab \cup bPc \cup cd \cup dPw.
        \]
        Then $P'$ is also an~$H$-extendable path.
Since~$G$ is $C_4$-subgraph-free, at least one of $aPc$ and $bPd$ has length greater than~$1$, and hence $P'$ is shorter than~$P$, a contradiction.
    \end{proof}

    \begin{figure}%
    \centering
    \begin{subfigure}{0.45\linewidth}
    \centering
    \begin{tikzpicture}
    \tikzstyle{B}=[circle,draw,fill=black!50, opacity=0, inner sep=0pt,minimum width=0pt]
    \draw (-2.15, 2.85) node [B] (V) {};
    \draw (-2.15, 1.15) node [B] (W) {};
      \draw[thick, blue, postaction={decorate}, 
      decoration={markings,
      mark =at position 0.2 with {\coordinate (A) at (0,0);},
      mark =at position 0.45 with {\coordinate (B) at (0,0);},
      mark =at position 0.55 with {\coordinate (C) at (0,0);},
      mark =at position 0.85 with {\coordinate (D) at (0,0);}
      }] 
      (V) to[out=-0.5, in=0.5, looseness=5.5] (W);

    \begin{scope}
        \clip (A) circle (1.48); 
        \clip (B) circle (1.35);
        \draw[line width = 3pt, white] (V) to[out=-0.5, in=0.5, looseness=5.5] (W);
    \end{scope}
    \begin{scope}
        \clip (C) circle (1.75); 
        \clip (D) circle (1.8); 
        \draw[line width = 3pt , white] (V) to[out=-0.5, in=0.5, looseness=5.5] (W);
    \end{scope}
    
    \draw[thick, blue] (A)--(B);
    \draw[thick, blue] (C)--(D);

    \tikzstyle{v}=[circle,draw,fill=black!50,inner sep=0pt,minimum width=4pt]
    \draw(-2.25,4.7) node [label=below:$H$]{};

    \draw[rounded corners, thick] (-2,0) -- (-1.5,0) -- (-1.5,4) --(-3,4)--(-3,0)--(-2,0);
      \draw (-2.15, 3) node [v] (v) {};
      \draw (-2.25, 1) node [v] (w) {};
      \draw[thick, postaction={decorate}, 
      decoration={markings,
      mark =at position 0.2 with {\coordinate (a) at (0,0);},
      mark =at position 0.45 with {\coordinate (b) at (0,0);},
      mark =at position 0.55 with {\coordinate (c) at (0,0);},
      mark =at position 0.85 with {\coordinate (d) at (0,0);}
      }] 
      (v) to[out=-0.5, in=0.5, looseness=5] (w);

    \draw[thick] (a)--(b);
    \draw[thick] (c)--(d);

    \draw (a) node[v] [label=above:$a$] {};
    \draw (b) node[v] [label=above:$b$] {};
    \draw (c) node[v] [label=below:$c$] {};
    \draw (d) node[v] [label=below:$d$] {};

\end{tikzpicture}
    \caption{The case when~$P$ has two internally disjoint chords. }
\end{subfigure}
\hfill
  \begin{subfigure}{0.45\linewidth}
  \centering
    \begin{tikzpicture}
    
    \tikzstyle{B}=[circle,draw,fill=black!50, opacity=0, inner sep=0pt,minimum width=0pt]
    \draw (-2.15, 3.17) node [B] (V) {};
      \draw (-2.15, 1.17) node [B] (W) {};
      \draw[thick, blue, postaction={decorate}, 
      decoration={markings,
      mark =at position 0.25 with {\coordinate (A) at (0,0);},
      mark =at position 0.6 with {\coordinate (B) at (0,0);},
      mark =at position 0.45 with {\coordinate (C) at (0,0);},
      mark =at position 0.8 with {\coordinate (D) at (0,0);}
      }] 
      (V) to[out=-0.5, in=0.5, looseness=5] (W);

    \begin{scope}
        \clip (A) circle (1.3); 
        \clip (B) circle (1.61);
        \draw[line width = 3pt, white] (V) to[out=-0.5, in=0.5, looseness=5] (W);
    \end{scope}
        \begin{scope}
        \clip (C) circle (1.91); 
        \clip (D) circle (1.29); 
        \draw[line width = 3pt , white] (V) to[out=-0.5, in=0.5, looseness=5] (W);
    \end{scope}

    \draw[thick, blue] (A)--(B);
    \draw[thick, blue] (C)--(D);

    \tikzstyle{v}=[circle,draw,fill=black!50,inner sep=0pt,minimum width=4pt]

    \draw[rounded corners, thick] (-2,0) -- (-1.5,0) -- (-1.5,4) --(-3,4)--(-3,0)--(-2,0);
      \draw(-2.25,4.7) node [label=below:$H$]{};
      \draw (-2.25, 3) node [v] (v) [label=left: $v$]{};
      \draw (-2.25, 1) node [v] (w) [label=left: $w$]{};
      \draw[thick, postaction={decorate},
      decoration={markings,
      mark =at position 0.25 with {\coordinate (a) at (0,0);},
      mark =at position 0.6 with {\coordinate (b) at (0,0);},
      mark =at position 0.45 with {\coordinate (c) at (0,0);},
      mark =at position 0.8 with {\coordinate (d) at (0,0);}
      }] 
      (v) to[out=-0.5, in=0.5, looseness=5] (w);

    \draw[thick] (a)--(b);
    \draw[thick] (c)--(d);

    \draw (a) node[v] [label=above:$a$] {};
    \draw (b) node[v] [label=below:$b$] {};
    \draw (c) node[v] [label=above:$c$] {};
    \draw (d) node[v] [label=below:$d$] {};
 
\end{tikzpicture}
\caption{The case when~$P$ has two chords $ab, cd$ with $c\in V(aPb)\setminus \{a,b\}$ and $d\notin V(aPb)$.}
\end{subfigure}
    \caption{A shortest~$H$-extendable path with two chords in Lemma~\ref{lem:shortestHpath-chordsofpath}. }
    \label{fig:chords2}
\end{figure}
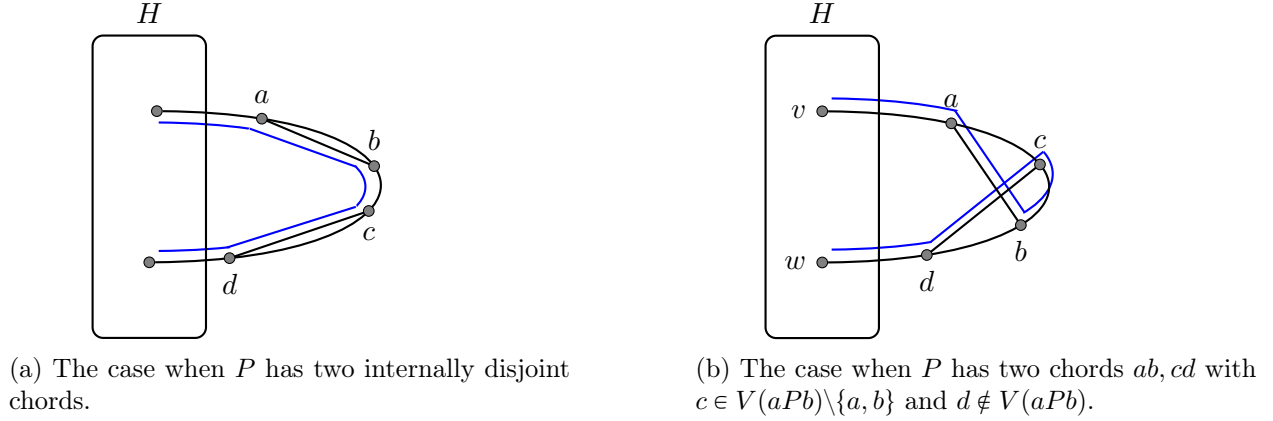

    \subsection{Chords whose one endpoint is in an~$H$-extendable path}

    Now, we deal with chords whose one endpoint is in an~$H$-extendable path. 
    We consider several possible configurations of a pair consisting of such a chord and other types of chords.

    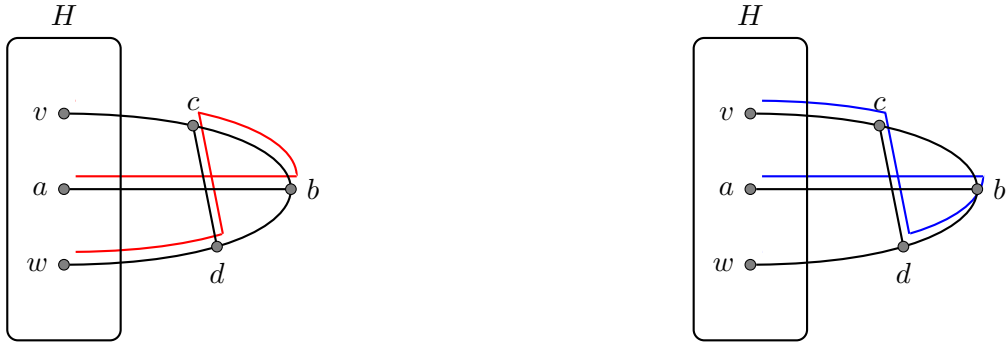
\begin{figure}[t]%
    \centering
    \begin{subfigure}{0.45\linewidth}
    \centering
    \begin{tikzpicture}
    \path[use as bounding box] (-3.2,-0.5) rectangle (1.5,4.8);
    \tikzstyle{R}=[circle, draw, fill=black!50, opacity=0, inner sep=0pt,minimum width=0pt]
    \draw (-2.12, 3.17) node [R] (v1) {};
      \draw (-2.12, 1.17) node [R] (w1) {};
      \draw (-2.12, 2.17) node [R] (a1) {};
      \draw[thick, red, postaction={decorate}, 
      decoration={markings,
      mark =at position 0.5 with {\coordinate (b1) at (0,0);},
      mark =at position 0.25 with {\coordinate (c1) at (0,0);},
      mark =at position 0.7 with {\coordinate (d1) at (0,0);}
      }] 
      (v1) to[out=-0.5, in=0.5, looseness=5] (w1);

    \begin{scope}
        \clip (v1) circle (1.65); 
        \clip (c1) circle (1.8);
        \draw[line width = 3pt, white] (v1) to[out=-0.5, in=0.5, looseness=5] (w1);
    \end{scope}
    \begin{scope}
        \clip (b1) circle (1.25); 
        \clip (d1) circle (1.25); 
        \draw[line width = 3pt , white] (v1) to[out=-0.5, in=0.5, looseness=5] (w1);
    \end{scope}
    
    \draw[thick, red] (a1)--(b1);
    \draw[thick, red] (c1)--(d1);

    \tikzstyle{v}=[circle,draw,fill=black!50,inner sep=0pt,minimum width=4pt]
    \draw(-2.25,4.7) node [label=below:$H$]{};
    \draw[rounded corners, thick] (-2,0) -- (-1.5,0) -- (-1.5,4) --(-3,4)--(-3,0)--(-2,0);
      \draw (-2.25, 3) node [v] [label=left:$v$]{};
      \draw (-2.25, 1) node [v] (w) [label=left:$w$]{};
      \draw (-2.25,2) node [v] (a) [label=left:$a$]{};
      \draw[thick, postaction={decorate}, 
      decoration={markings,
      mark =at position 0.5 with {\coordinate (b) at (0,0);},
      mark =at position 0.25 with {\coordinate (c) at (0,0);},
      mark =at position 0.7 with {\coordinate (d) at (0,0);}
      }] 
      (v) to[out=-0.5, in=0.5, looseness=5] (w);

    \draw[thick] (a)--(b);
    \draw[thick] (c)--(d);

    \draw (b) node[v] [label = right: $b$]{};
    \draw (c) node[v] [label = above: $c$]{};
    \draw (d) node[v] [label = below: $d$]{};

\end{tikzpicture}
\end{subfigure}
\hfill
  \begin{subfigure}{0.45\linewidth}
  \centering
    \begin{tikzpicture}
    \path[use as bounding box] (-3.2,-0.5) rectangle (1.5,4.8);
    \tikzstyle{B}=[circle,draw,fill=black!50, opacity=0, inner sep=0pt,minimum width=0pt]
    \draw (-2.12, 3.17) node [B] (V) {};
      \draw (-2.12, 1.17) node [B] (W) {};
      \draw (-2.12, 2.17) node [B] (A) {};
      \draw[thick, blue, postaction={decorate}, 
      decoration={markings,
      mark =at position 0.5 with {\coordinate (B) at (0,0);},
      mark =at position 0.25 with {\coordinate (C) at (0,0);},
      mark =at position 0.7 with {\coordinate (D) at (0,0);}
      }] 
      (V) to[out=-0.5, in=0.5, looseness=5] (W);

    \begin{scope}
        \clip (C) circle (1.55); 
        \clip (B) circle (1.54);
        \draw[line width = 3pt, white] (V) to[out=-0.5, in=0.5, looseness=5] (W);
    \end{scope}
    \begin{scope}
        \clip (W) circle (2); 
        \clip (D) circle (2); 
        \draw[line width = 3pt , white] (V) to[out=-0.5, in=0.5, looseness=5] (W);
    \end{scope}
    
    \draw[thick, blue] (A)--(B);
    \draw[thick, blue] (C)--(D);

    \tikzstyle{v}=[circle,draw,fill=black!50,inner sep=0pt,minimum width=4pt]
    \draw(-2.25,4.7) node [label=below:$H$]{};
    \draw[rounded corners, thick] (-2,0) -- (-1.5,0) -- (-1.5,4) --(-3,4)--(-3,0)--(-2,0);
      \draw (-2.25, 3) node [v] [label=left:$v$]{};
      \draw (-2.25, 1) node [v] (w) [label=left:$w$]{};
      \draw (-2.25,2) node [v] (a) [label=left:$a$]{};
      \draw[thick, postaction={decorate}, 
      decoration={markings,
      mark =at position 0.5 with {\coordinate (b) at (0,0);},
      mark =at position 0.25 with {\coordinate (c) at (0,0);},
      mark =at position 0.7 with {\coordinate (d) at (0,0);}
      }] 
      (v) to[out=-0.5, in=0.5, looseness=5] (w);

    \draw[thick] (a)--(b);
    \draw[thick] (c)--(d);
    
    \draw (b) node[v] [label = right: $b$]{};
    \draw (c) node[v] [label = above: $c$]{};
    \draw (d) node[v] [label = below: $d$]{};

\end{tikzpicture}
\end{subfigure}
    \caption{A shortest~$H$-extendable path with two chords in Lemma~\ref{lem:shortestHpath-chordandedge}.
The red and blue paths denote $P_1$ and $P_2$ in the proof. }
    \label{fig:chords3}
\end{figure}

\begin{lemma}\label{lem:shortestHpath-chordandedge}
    Let~$G$ be a graph and let~$H$ be a connected bipartite subgraph of~$G$ and $S \subseteq V(H)$.
Let~$P$ be a shortest path among~$H$-extendable paths in $G-N_{G-(V(H)\setminus S)}[S]$.
If~$P$ has length at least~$5$, then $H \cup P$ has no two chords $ab$ and $cd$ such that
    \begin{itemize}
        \item $a\in V(H)\setminus V(P)$ and $b, c, d \in V(P)\setminus V(H)$, and 
        \item~$b$ is an internal vertex of the subpath $cPd$.
    \end{itemize}
\end{lemma}
    
    \begin{proof}
        Let~$v$ and~$w$ be the endpoints of~$P$ such that $v,c,b,d,w$ appear in this order when traversing~$P$ from~$v$ to~$w$.
See Figure~\ref{fig:chords3} for an illustration. 
        
        Let $\beta$ be a bipartition indicator of~$H$. 
        Since~$P$ is a shortest~$H$-extendable path and $vPb\cup ba$ is strictly shorter than~$P$, 
        $H\cup (vPb\cup ab)$ is not bipartite.
Then, we have that $\abs{E(vPb)} + 1 \equiv \beta(v) + \beta(a) +1 \pmod{2}$.
        Similarly, we have that 
        $\abs{E(bPw)} + 1 \equiv \beta(a) + \beta(w) + 1 \pmod{2}$.

        By Lemma~\ref{lem:onechordforpath}, $\abs{E(cPd)} \equiv 0 \pmod{2}$.
Then, we obtain 
        \begin{align*}
            |E(vPc)| + |E(bPd)| &\equiv |E(vPb)| + |E(cPd)|
                                \equiv \beta(v) + \beta(a) \pmod{2} \\
            |E(cPb)| + |E(dPw)| &\equiv |E(bPw)| + |E(cPd)| 
                                \equiv \beta(a) + \beta(w) \pmod{2}.
        \end{align*}

        Let $P_1 = ab \cup bPc \cup cd \cup dPw$ and $P_2 = vPc \cup cd \cup dPb \cup ba$.
        The above congruences imply that both $P_1$ and $P_2$ are~$H$-extendable paths. 
        
        As~$P$ has length at least~$5$, one of $vPc$, $cPb$, $bPd$, and $dPw$ has length greater than 1.
So, there exists $i \in \{1,2\}$ such that $\abs{E(P_i)} < |E(P)|$, a contradiction.
        \end{proof}

  \begin{lemma}\label{lem:chordandedge}
    Let~$G$ be a graph and let~$H$ be a connected bipartite subgraph of~$G$ and $S \subseteq V(H)$.
Let~$P$ be a shortest path among~$H$-extendable paths in $G-N_{G-(V(H)\setminus S)}[S]$.
Then $H\cup P$ has no two chords $ab$ and $cd$ such that $a\in V(H)\setminus V(P)$, $b\in V(P)\setminus V(H)$, $c, d\in V(P)$, and~$b$ is not an internal vertex of $cPd$.
    \end{lemma}
    \begin{proof}
       Assume, for the sake of contradiction, that $H\cup P$ has two such chords~$ab$ and~$cd$.
       Let~$v$ and~$w$ be the endpoints of~$P$, and without loss of generality, we may assume that $v,b,c,d,w$ appear in this order when we traverse from~$v$ to~$w$ in~$P$, possibly $b=c$.
       
       Let $\beta$ be a bipartition indicator of~$H$.
By Lemma~\ref{lem:indicatornot}, $\abs{E(P)} \equiv \beta(v) + \beta(w) \pmod{2}$.
We know that both $vPb \cup ba$ and $ab \cup bPc \cup cd \cup dPw$ are~$H$-paths shorter than~$P$.
As~$P$ is a shortest path among~$H$-extendable paths, they are not~$H$-extendable.
       So, by Lemma~\ref{lem:indicatornot}, we have
       \begin{align*}
           &1+|E(vPb)| \equiv \beta(v) + \beta(a) + 1 \pmod{2} \\
           &2+|E(bPc)| + |E(dPw)| \equiv \beta(a) + \beta(w) + 1\pmod{2}.
       \end{align*}
        By adding these two congruences, we get 
        $\abs{E(vPb)}+|E(bPc)| + |E(dPw)| \equiv \beta(v) + \beta(w) + 1 \pmod{2}$.
By Lemma~\ref{lem:onechordforpath}, $\abs{E(cPd)} \equiv 0 \pmod 2$.
Thus, \[|E(P)| \equiv |E(vPb)|+|E(bPc)| + |E(cPd)| + |E(dPw)| \equiv  \beta(v) + \beta(w) + 1\pmod{2},\] which is a contradiction.
    \end{proof}

    \begin{lemma}\label{lem:shorthestHpath-edges}
    Let~$G$ be a graph and let~$H$ be a connected bipartite subgraph of~$G$ and $S \subseteq V(H)$.
Let $P=p_1p_2 \cdots p_m$ be a shortest path among~$H$-extendable paths in $G-N_{G-(V(H)\setminus S)}[S]$. 

    If $H\cup P$ has two chords $a_1b_1$ and $a_2b_2$ such that $a_i\in V(H)\setminus V(P)$ and $b_i\in V(P)\setminus V(H)$ for each $i\in\{1,2\}$, where $a_1\neq a_2$ and $b_1\neq b_2$, then $\{b_1,b_2\}=\{p_2,p_{m-1}\}$.
    \end{lemma}
    \begin{proof}
    Assume that $H\cup P$ has two such chords $a_1b_1$ and $a_2b_2$ with $\{b_1,b_2\} \neq \{p_2,p_{m-1}\}$.
Let~$v$ and~$w$ be the endpoints of~$P$ so that $v,b_1,b_2,w$ appear in this order when we traverse from~$v$ to~$w$ in~$P$.
Let $P_1 =vPb_1 \cup a_1b_1$, $P_2 =a_2b_2 \cup b_2Pw$, and $P_3 =a_1b_1 \cup b_1Pb_2 \cup a_2b_2$.
Let $\beta$ be a bipartition indicator of~$H$. 
    
    We observe that $P_1$ and $P_2$ are shorter than~$P$.
Furthermore, because of the condition $\{b_1, b_2\}\neq \{p_2, p_{m-1}\}$, $P_3$ is shorter than~$P$.
Since~$P$ is a shortest~$H$-extendable path, $H \cup P_i$ is not bipartite for all $i \in \{1,2,3\}$.
        By Lemma~\ref{lem:indicatornot}, it follows that
        \begin{align*}
            &|E(vPb_1)| + 1 \equiv \beta(v) + \beta(a_1) + 1 \pmod{2}, \\
            &|E(b_2Pw)| + 1 \equiv \beta(w) + \beta(a_2) + 1 \pmod{2}, \\
            &|E(b_1Pb_2)| + 2 \equiv \beta(a_1) + \beta(a_2) + 1 \pmod{2}.
        \end{align*}
         By adding these congruences, we obtain
         \begin{align*}
             |E(P)| &\equiv (\beta(v) + \beta(a_1) + 1) + (\beta(w) + \beta(a_2) + 1) + (\beta(a_1) + \beta(a_2) + 1) \\
             &\equiv \beta(v) + \beta(w) + 1 \pmod{2}.
         \end{align*}
         However, we know that $\abs{E(P)} \equiv \beta(v) + \beta(w) \pmod{2}$ as $H\cup P$ is bipartite.
This is a contradiction.
     \end{proof}

    \begin{lemma}\label{lem:shorthestHpath-edges-2}
    Let~$G$ be a graph and let~$H$ be a connected bipartite subgraph of~$G$ and $S \subseteq V(H)$.
Let~$P$ be a shortest path among~$H$-extendable paths in $G-N_{G-(V(H)\setminus S)}[S]$, and suppose that~$P$ has length at least~$3$.
    Then, $H \cup P$ has no two chords $a_1b$ and $a_2b$ such that $a_1, a_2\in V(H)\setminus V(P)$ with $a_1\neq a_2$ and $b\in V(P)\setminus V(H)$.
    \end{lemma}
\begin{proof}
    Suppose that $a_1b\in E(G)$ and $a_2b\in E(G)$.
Let $\beta$ be a bipartition indicator of~$H$.
Let~$v$ and~$w$ be the endpoints of~$P$. 
    
    Since~$P$ has length at least~$3$, $(a_1ba_2) \cup H$ is not bipartite.
This implies that $a_1$ and $a_2$ lie in different parts of the bipartition.
Thus, $\beta(a_1)+\beta(a_2) \equiv 1 \pmod{2}$. 
    
    We claim that both $vPb$ and $bPw$ have length at least~$2$.
If at least one of $vPb$ and $bPw$ has length~$1$, adding either $a_1b$ or $a_2b$ yields a shorter~$H$-extendable path than~$P$.
This yields a contradiction.
Therefore, for $v' \in \{v, w\}$ and $i \in \{1, 2\}$, $v'Pb \cup ba_i$ forms an~$H$-path shorter than~$P$, which is not~$H$-extendable.
    By Lemma~\ref{lem:indicatornot}, we attain
    \begin{align*}
        &\beta(v) + \beta(a_1) \equiv 1 + (|E(vPb)| + |a_1b|) \pmod{2} \\
        &\beta(v) + \beta(a_2) \equiv 1 + (|E(vPb)| + |a_2b|) \pmod{2}.
    \end{align*}
    Adding the above congruences, we have $\beta(a_1) + \beta(a_2) \equiv  0 \pmod{2}$, which is a contradiction.
\end{proof}

 \begin{lemma}\label{lem:shorthestHpath-edges-3}
    Let~$G$ be a graph and let~$H$ be a connected bipartite subgraph of~$G$ and $S \subseteq V(H)$.
Let~$P$ be a shortest path among~$H$-extendable paths in $G-N_{G-(V(H)\setminus S)}[S]$.
    
    Let $a\in V(H)\setminus V(P)$ and $b_1, b_2\in V(P)\setminus V(H)$ with $b_1\neq b_2$ such that $ab_1, ab_2\in E(G)$.
Then there is no edge $a'b'\in E(G)$ with $a'\in V(H)\setminus (V(P)\cup \{a\})$ and $b'\in V(P)\setminus V(H)$.
\end{lemma}
    
\begin{proof}
    Suppose that there is an edge $a'b'\in E(G)$ with $a'\in V(H)\setminus (V(P)\cup \{a\})$ and $b'\in V(P)\setminus V(H)$.
If there exists $i \in \{1, 2\}$ such that $b'=b_i$, then this contradicts Lemma~\ref{lem:shorthestHpath-edges-2}.
Otherwise, $b'\notin \{b_1, b_2\}$.
Since $b_1\neq b_2$, there exists $i \in \{1,2\}$ such that $\{b', b_i\} \neq \{p_2, p_{m-1}\}$.
Then this contradicts  Lemma~\ref{lem:shorthestHpath-edges}.
\end{proof}

\begin{lemma}\label{lem:nestednew}
    Let~$G$ be a $C_4$-subgraph-free graph, and let~$C$ be a shortest even cycle in~$G$ such that $G-N_G[C]$ has no even cycle.
Let~$H$ be a connected bipartite subgraph of~$G$ containing~$C$, and let $S \subseteq V(H)$.
Let~$P$ be a shortest~$H$-extendable path in $G-N_{G-(V(H)\setminus S)}[S]$.
    Then~$P$ has no three pairwise vertex-disjoint chords whose endpoints are all internal vertices of~$P$.
\end{lemma}
\begin{proof}
    Let $P=p_1p_2\cdots p_m$.
    By Lemma~\ref{lem:shortestHpath-chordsofpath}, for two chords~$ab$ and~$cd$ of~$P$ where $\{a,b\}\cap \{c,d\}=\emptyset$, either $V(aPb)\subseteq V(cPd)$ or $V(cPd)\subseteq V(aPb)$.

    Suppose that there are three chords $a_1b_1, a_2b_2, a_3b_3$ such that $V(a_3Pb_3)\subseteq V(a_2Pb_2)\subseteq V(a_1Pb_1)$, and assume that $a_1, a_2, a_3, b_3, b_2, b_1$ appear in this order in~$P$.

    Let $D=a_2Pa_3\cup a_3b_3\cup b_3Pb_2\cup b_2a_2$.
By Lemma~\ref{lem:onechordforpath}, 
    $\abs{E(a_2Pb_2)}\equiv |E(a_3Pb_3)|\equiv0\pmod2$,
    and thus, $D$ is an even cycle.

    Because~$D$ consists of internal vertices of~$P$ and~$C$ is a subgraph of~$H$, $C$ and~$D$ are vertex-disjoint. 
    As $G-N_G[C]$ has no even cycle, there is an edge $xy$ with $x\in V(C)$ and $y\in V(D)$.

    First assume that $x\notin V(P)$.
By the choice of chords, $y\notin V(H)$.
As~$P$ contains three vertex-disjoint chords, $P$ has length at least~$5$.
So, by Lemmas~\ref{lem:shortestHpath-chordandedge} and~\ref{lem:chordandedge}, this is not possible.

    Thus, we may assume that $x\in V(P)$.
    Since $\{a_1, b_1\}$ separates $\{a_2,b_2\}$ and the endpoints of~$P$ in~$P$, $xy$ is not an edge of~$P$.
Thus, $xy$ is a chord of~$P$.
    Then $xy$ and $a_1b_1$ contradict (2) of Lemma~\ref{lem:shortestHpath-chordsofpath}.
\end{proof}

\subsection{Proof of Proposition~\ref{prop:boundingchord}}

We now prove Proposition~\ref{prop:boundingchord} using the structural properties of chords established in the previous subsections.

\begin{lemma}\label{lem:forestvc}
    Let~$J$ be a forest and~$Z$ be the set of vertices of degree~$1$ in~$J$ whose unique neighbor has also degree~$1$ in~$J$.
If~$k$ is the maximum size of a matching in~$J$, then $\abs{V_{\ge 2}(J)\cup Z}\le 2k$.
\end{lemma}
\begin{proof}
    Let~$r$ be the number of components of~$J$ isomorphic to $K_2$, and let $J'$ be obtained from~$J$ by deleting these components.
Then $\abs{Z}=2r$ and the maximum matching size of $J'$ is $k-r$.

Since $J'$ is bipartite, by Kőnig's theorem (see \cite[Theorem 2.1.1]{Diestel2025}), it has a vertex cover~$X$ of size $k-r$.
Let
$A=V_{\ge 2}(J')\setminus X$.
As~$X$ is a vertex cover, every neighbor of a vertex in~$A$ belongs to~$X$.
Thus, in the forest consisting of the edges between~$A$ and~$X$, every vertex of~$A$ has degree at least~$2$.
Therefore, the number of edges between~$A$ and~$X$ is at least~$2$.
On the other hand, as $J'$ is a forest, the number of edges between~$A$ and~$X$ is at most $\abs{A}+|X|$.
Thus we have $\abs{A}\le |X|$.
Consequently,
\[
|V_{\ge 2}(J')|
\le |A|+|X|
\le 2|X|
=2(k-r).
\]
Together with $\abs{Z}=2r$, this gives
$\abs{V_{\ge 2}(J)\cup Z}\le 2(k-r)+2r=2k$.
\end{proof}

\begin{proof}[Proof of Proposition~\ref{prop:boundingchord}]
    (1) Let $w\in V_{\ge 2}(J)$.
Then~$w$ is incident with two edges in~$M$, say~$wz_1$ and~$wz_2$.
If $w\in V(H)$, then $z_1, z_2\in V(\internal(P))$ and $w\notin V(z_1Pz_2)$.
Thus, we may assume that $w\in V(\internal(P))$.

    Assume that $z_1, z_2\in V(P)$.
Then by (1) of Lemma~\ref{lem:shortestHpath-chordsofpath}, $w\notin V(z_1Pz_2)$, as desired.
So, we may assume that $z_1\notin V(P)$.
By Lemma~\ref{lem:shorthestHpath-edges-2}, $z_2$ is not contained in $V(H)\setminus V(P)$ while $z_2$ is not contained in $V(P)$ by Lemma~\ref{lem:chordandedge}.
This is a contradiction.
Thus, this case is not possible.

    \medskip

    (2) Let $M_P$ be the set of chords in~$M$ whose both endpoints are contained in~$P$, 
    and $M_H=M\setminus M_P.$
Thus, every edge in $M_H$ has one endpoint in
$V(H)\setminus V(P)$ and the other in $V(\internal(P))$.

We first consider $M_P$.
By
Lemma~\ref{lem:shortestHpath-chordsofpath}, the subpaths of~$P$
spanned by any two edges in $M_P$ are nested.
Indeed, if the two
chords are vertex-disjoint, then the other two possible configurations
are excluded by (1) and (2) of
Lemma~\ref{lem:shortestHpath-chordsofpath}.
If they share an endpoint,
their other endpoints must lie on the same side of the common
endpoint by (1) of the same lemma.

Let $J_P$ be the graph with vertex set $V(P)$ and edge set $M_P$.
We claim that $J_P$ is a forest.
Suppose otherwise, and choose an edge
$xy$ of a cycle of $J_P$ for which $xPy$ is minimal with respect to
the inclusion relation.
Let $xx'$ and $yy'$ be the other two edges of this cycle
incident with~$x$ and~$y$, respectively.
By the minimality of $xPy$
and the nesting property, both $xPx'$ and $yPy'$ properly contain
$xPy$.
But then the subpaths $xPx'$ and $yPy'$ are not
nested, a contradiction.

We next claim that every matching of $J_P$ has size at most~$3$.
Let~$F$ be a matching in $J_P$.
Let $P=p_1p_2\cdots p_m$.
At most one edge of~$F$ has an
endpoint in $\{p_1,p_m\}$.
Indeed, two such vertex-disjoint chords,
one incident with $p_1$ and the other with $p_m$, cannot be nested,
contrary to Lemma~\ref{lem:shortestHpath-chordsofpath}.
All other edges of~$F$ have both endpoints in $\internal(P)$.
By Lemma~\ref{lem:nestednew}, there are at most two such edges.
Therefore, $\abs{F}\le 3$, proving the claim.

Let $Z_P$ be the set of vertices of degree~$1$ in $J_P$ whose unique neighbor also has degree~$1$ in $J_P$.
By Lemma~\ref{lem:forestvc}, 
$V_{\ge 2}(J_P)\cup Z_P$ is a vertex cover of $J_P$ of size at most~$6$.
Let $T_P=V_{\ge 2}(J_P)\cup Z_P$.

We now consider $M_H$.
By
Lemma~\ref{lem:shorthestHpath-edges-2}, every vertex of
$\internal(P)$ is incident with at most one edge in $M_H$.
If some vertex $a\in V(H)\setminus V(P)$ is incident with at least
two edges in $M_H$, then
Lemma~\ref{lem:shorthestHpath-edges-3} implies that every edge in
$M_H$ is incident with~$a$.
In this case, we set $T_H=\{a\}$.

Otherwise, $M_H$ is a matching.
By
Lemma~\ref{lem:shorthestHpath-edges}, it has at most two edges.
So, in this case, we set $T_H$ as the set of all vertices incident with $M_H$.
In either case,   $\abs{T_H}\le 4$,
and $T_H$ contains every vertex of $V(H)\setminus V(P)$ incident with
at least two edges in $M_H$.

Furthermore, by Lemma~\ref{lem:chordandedge}, no vertex of $\internal(P)$ is incident
with both an edge of $M_H$ and an edge of $M_P$.
Together with
Lemma~\ref{lem:shorthestHpath-edges-2}, this implies that every vertex
of $V_{\ge 2}(J)\cap V(P)$ is incident with at least two edges in $M_P$.
Thus, we have 
    $V_{\ge 2}(J)\subseteq T_P\cup T_H$.

We next show that $Z\subseteq T_P\cup T_H$.
Let $z\in Z$, and let $zz'$ be the unique edge of~$M$ incident with~$z$.
Recall that $z'$ also has degree~$1$ in~$J$.

Suppose first that $zz'\in M_P$.
Since both~$z$ and $z'$ have degree
$1$ in~$J$, they also have degree~$1$ in $J_P$.
Thus,
$z\in Z_P\subseteq T_P$.

Now suppose that $zz'\in M_H$.
If some vertex
$a\in V(H)\setminus V(P)$ is incident with at least two edges in
$M_H$, then every edge in $M_H$ is incident with~$a$.
In this case,
no endpoint of an edge in $M_H$ belongs to~$Z$: the vertex~$a$ has
degree at least~$2$ in~$J$, while every other endpoint of an edge in
$M_H$ is adjacent to~$a$, which has degree at least~$2$ in~$J$.
Thus, this case cannot occur for $z\in Z$.

Therefore, $M_H$ is a matching.
By the definition of $T_H$ in this
case, $T_H$ contains all endpoints of the edges in $M_H$.
Hence
$z\in T_H$.
Therefore, we have $Z\subseteq T_P\cup T_H.$

Set $T=T_P\cup T_H.$
Then~$T$ is a vertex cover of~$J$, and
$V_{\ge 2}(J)\cup Z\subseteq T$.
Moreover, we have $\abs{T}\le |T_P|+|T_H|\le 6+4=10.$  Note that $V_{\ge 2}(J)\cup Z$ is a vertex cover of~$J$.
This proves (2).
\end{proof}

\section{Skeletons}\label{sec:skeleton}

In this section, we introduce skeletons, which are bipartite subgraphs with a controlled chord structure, and develop the main tools for extracting induced packings of even cycles from them.
We show that a skeleton with sufficiently many branching vertices contains a large induced packing of even cycles.
A key ingredient in the proof is a reduction that modifies the skeleton when a branching path meets many chords, while preserving the structural properties needed for the argument.

For a graph~$G$, a bipartite $(2,4)$-subgraph~$H$ of~$G$ is a \emph{semi-skeleton} in~$G$ if 
\begin{itemize}
   \item the set of chords of~$H$ is a matching in~$G$, 
    \item no chord of~$H$ has both endpoints at branching vertices of~$H$, and
    \item no vertex of degree~$4$ in~$H$ is incident with a chord of~$H$. 
\end{itemize}
A semi-skeleton~$H$ is a \emph{skeleton} 
if every chord of~$H$ is parity-breaking for~$H$.

We define the \emph{branching excess} of a skeleton $H$ by
\[
\epsilon(H)=\sum_{v\in V_{\ge3}(H)}(\deg_H(v)-2).
\]

For each corridor~$R$ of a semi-skeleton~$H$, we fix an arbitrary orientation of~$R$, which induces a linear order $<_{R}$ on $V(R)$. 
We observe the following.

\begin{lemma}\label{lem:skeletoncomponent}
    Let~$H$ be a skeleton of~$G$.
Then there is no chord of~$H$ incident with distinct connected components of~$H$.
\end{lemma}
\begin{proof}
    Suppose that a chord~$e$ of~$H$ has its endpoints in distinct connected components of~$H$.
    Then $H+e$ is bipartite, as we may reverse the bipartition of one of the two components if necessary.
    Thus, $e$ is parity-preserving for~$H$, contradicting the definition of a skeleton.
\end{proof}

An induced $(2,4)$-subgraph of a skeleton is in general only a semi-skeleton, since a remaining chord whose endpoints lie in distinct components of the subgraph become parity-preserving.
To turn a semi-skeleton into a skeleton, we repeatedly add parity-preserving chords.
The following observation will be useful.

\begin{lemma}\label{lem:semiskeleton-ppchord}
    Let~$G$ be a graph and~$H$ be a semi-skeleton in~$G$.
If~$e$ is a parity-preserving chord of~$H$, then $H+e$ is a semi-skeleton.
\end{lemma}
\begin{proof}
Let $e=xy$ be a parity-preserving chord of~$H$, and let $H'=H+e$.
Since~$e$ is parity-preserving for~$H$, the graph $H'$ is bipartite.

We first verify that $H'$ is a $(2,4)$-subgraph of~$G$.
Since no vertex of degree~$4$ in~$H$ is incident with a chord of~$H$,
we have $\deg_H(x),\deg_H(y)\leq 3$.
Thus, adding~$e$ increases the degrees of~$x$ and~$y$ by one, and
every vertex of $H'$ has degree between~$2$ and~$4$.

Note that the set of chords of $H'$ is precisely the set of chords
of~$H$ other than~$e$.
Hence, the chords of $H'$ form a matching.

We next show that no vertex of degree~$4$ in $H'$ is incident with
a chord of $H'$.
Let~$v$ be a vertex of degree~$4$ in $H'$.
If $\deg_H(v)=4$, then~$v$ is not incident with a chord of~$H$,
and hence not with a chord of $H'$.
Otherwise, $v$ is one of
$x$ and~$y$.
Since the chords of~$H$ form a matching and~$v$ is
incident with~$e$, it is not incident with any other chord of~$H$.
Thus, $v$ is not incident with a chord of $H'$.

Finally, suppose that a chord $uv$ of $H'$ has both endpoints at
branching vertices of $H'$.
Since the chords of~$H$ form a matching,
neither~$u$ nor~$v$ is an endpoint of~$e$.
Therefore,
\[
    \deg_{H'}(u)=\deg_H(u)
    \quad\text{and}\quad
    \deg_{H'}(v)=\deg_H(v).
\]
Hence, $u$ and~$v$ are also branching vertices of~$H$, contradicting
the fact that~$H$ is a semi-skeleton.

Therefore, $H'$ is a semi-skeleton in~$G$.
\end{proof}

A \emph{skeleton completion} of a semi-skeleton~$H$ is a skeleton obtained from~$H$ by repeatedly adding parity-preserving chords.
This can be obtained by applying Lemma~\ref{lem:semiskeleton-ppchord} repeatedly until the graph has no parity-preserving chords.
Thus, every semi-skeleton has a skeleton completion. 

We now classify the chords of a skeleton~$H$ into three types.
\begin{itemize}
    \item Let $F_0(H)$ be the set of chords $xy$ for which there exists a corridor~$P$ of~$H$ such that exactly one of~$x$ and~$y$ is an internal vertex of~$P$.
    \item Let $F_1(H)$ be the set of chords $xy$ whose both endpoints are internal vertices of the same corridor~$P$ such that there is a chord of~$H$ with exactly one endpoint in $V(xPy)$.
    \item Let $F_2(H)$ be the set of chords $xy$ whose both endpoints are internal vertices of the same corridor~$P$ such that there is no chord of~$H$ with exactly one endpoint in $V(xPy)$.
\end{itemize}

\begin{figure}[t]
    \centering
    \begin{tikzpicture}[scale=0.9,
    vertex/.style={circle, fill=black, inner sep=2pt}
]
    \draw[thick, rounded corners]
        (-4,-1.2) -- (-2,-1.2) -- (-1.6,0)
        -- (-2,1.2) -- (-4,1.2) -- (-4.5,0) -- cycle;

    \draw[thick, rounded corners]
        (4,-1.2) -- (2,-1.2) -- (1.6,0)
        -- (2,1.2) -- (4,1.2) -- (4.5,0) -- cycle;

    \node[vertex] (x) at (-1.6,0) {};
    \node[vertex] (y) at (1.6,0) {};
    \draw[very thick] (x) -- (y)
        node[midway,below] {$R$};

    \node[vertex] (a) at (-2,1.2) {};
    \node[vertex] (b) at (2,1.2) {};
    \node[above] at (a) {$a$};
    \node[above] at (b) {$b$};

    \draw[thick, blue]
        (a) to[bend left=22] node[midway,above] {$ab$} (b);
\end{tikzpicture}
\caption{In the construction of $\phi(H,R)$, the chord $ab$ is parity-breaking for~$H$.
However, after deleting $V(\internal(R))\setminus V(v_2Rv_3)$, the vertices~$a$ and~$b$ may lie in distinct components of $H'$, in which case $ab$ becomes parity-preserving for $H'$.
We then add $ab$ to $H'$, so that it is no longer a chord of $\phi(H,R)$. }
    \label{fig:skeletoncomponent}
\end{figure}
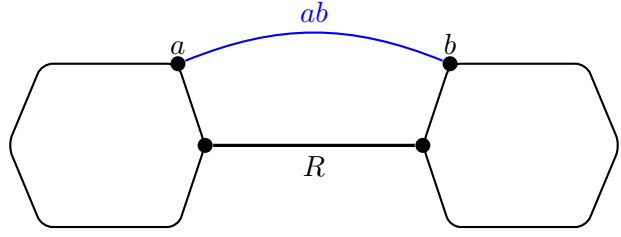

For a corridor~$R$ of~$H$ incident with at least~$4$ chords in $F_0(H)$, we define $\phi(H,R)$ as the graph obtained as follows.
\begin{itemize}
    \item Let $e_1, e_2, \ldots, e_m$ be the chords of $F_0(H)$ incident with~$R$.
For each $i\in [m]$, let $e_i=u_iv_i$ with $v_i\in V(\internal(R))$, such that $v_1<_{R}v_2<_{R}\cdots<_{R}v_m$. 
    \item Let $H'= \left(H- \big(V(\internal(R))\setminus V(v_{2}Rv_{3})\big)\right)+\{u_2v_2,u_3v_3\}$.
    In other words, $H'$ is the graph obtained from $H-V(\internal(R))$ by adding the path $u_{2}v_{2}\cup v_{2}Rv_{3}\cup v_{3}u_{3}$.
    \item Note that $H'$ may have a parity-preserving chord, see Figure~\ref{fig:skeletoncomponent}.
We define $\phi(H,R)$ to be the graph obtained from $H'$ by adding all parity-preserving chords of $H'$ in~$G$.
\end{itemize}

We show that $\phi(H,R)$ is again a skeleton with $\abs{F_0(\phi(H,R))}+\abs{F_1(\phi(H,R))}<\abs{F_0(H)}+\abs{F_1(H)}$.

\begin{lemma}\label{lem:phi-bipartite}
    Let~$G$ be a graph and~$H$ be a skeleton in~$G$.
Let~$R$ be a corridor of~$H$ incident with at least~$4$ chords in $F_0(H)$.
Then the following hold.
    \begin{enumerate}[label=(\arabic*)]
        \item $\phi(H,R)$ is skeleton.
        \item $\abs{F_0(\phi(H,R))}+\abs{F_1(\phi(H,R))}<\abs{F_0(H)}+\abs{F_1(H)}$.
    \end{enumerate}
\end{lemma}
\begin{proof}
     (1) Let $K=\phi(H,R)$.
     Assume that vertices and edges are labeled as in the construction of~$K$.
Also, let $H'$ be the graph $\left(H- \big(V(\internal(R))\setminus V(v_{2}Rv_{3})\big)\right)+\{u_2v_2,u_3v_3\}$. 
     Let~$p$ and~$q$ be the endpoints of~$R$.
Let~$C$ be the connected component of~$H$ containing~$R$.

      As~$u_2v_2$ and~$u_3v_3$ are parity-breaking chords, $u_2v_2\cup v_2Rv_3\cup v_3u_3$ is a parity-preserving path for $H-V(\internal(R))$.
So, $H'$ is bipartite.

     We claim that every chord of $H'$ whose endpoints lie in the same component of $H'[V(H')\cap V(C)]$ is parity-breaking for $H'$. 

     \begin{claim}\label{claim:pbchord}
         Every chord of $H'$ whose endpoints lie in the same component of $H'[V(H')\cap V(C)]$ is parity-breaking for $H'$. 
     \end{claim}
    \begin{clproof}
    Except for~$u_2v_2$ and~$u_3v_3$, which are edges of $H'$, every chord of $H'$ is a chord of~$H$.
If both endpoints of such a chord lie outside $v_2Rv_3$, then their parts are not changed, and the endpoints lie in $v_2Rv_3$, then their parts are reversed.
Thus the claim holds.
    \end{clproof}
     
    Now, suppose that $H'$ has a parity-preserving chord. 
    Note that there is no edge between $V(C)$ and $V(H)\setminus V(C)$ by Lemma~\ref{lem:skeletoncomponent}.
    Also, the construction does not change any component of~$H$ other than~$C$.
Thus, every parity-preserving chord of $H'$ has both endpoints in $V(C)$.

    By Claim~\ref{claim:pbchord}, such a chord must join two distinct components of $H'[V(H')\cap V(C)]$.
Since every internal vertex of~$R$ has degree~$2$ in~$H$, the graph $C-V(\internal(R))$ has at most two components.
Thus, it has exactly two components, say $C_1$ and $C_2$.
    Moreover, the path  $u_2v_2\cup v_2Rv_3\cup v_3u_3$ is added to one of $C_1$ and $C_2$.
    Without loss of generality, we assume that $u_2, u_3\in V(C_1)$.

    By the choice of $u_2$ and $u_3$, there is no chord of $F_0(H)$ incident with $\internal(v_2Rv_3)$.
Therefore, every parity-preserving chord of $H'$ has one endpoint in $V(C_1)$ and the other in $V(C_2)$. 
    
    In~$H$, the chords of~$H$ connecting $V(C_1)$ and $V(C_2)$ are all parity-breaking.
Thus, with respect to a fixed bipartition of~$H$, the two endpoints of every such chord lie in the same part.
Reversing the bipartition of $C_2$ makes all these chords join distinct parts.
    Therefore, adding all chords between $C_1$ and $C_2$ to $H'$ still preserve the bipartiteness. 
    This shows that~$K$ is bipartite.

    Since all parity-preserving chords of $H'$ are absorbed into~$K$, every chord of~$K$ has both endpoints in the same component of $H'[V(H')\cap V(C)]$.
Hence, by Claim~\ref{claim:pbchord}, every chord of~$K$ is parity-breaking.
    
    Note that every chord of~$K$ is a chord of $H'$, and also a chord of~$H$.
So, the set of chords of~$K$ is a matching in~$G$.

      As~$u_2$ and~$u_3$ are incident with some chord, $u_2$ and~$u_3$ are not vertices of degree~$4$ in~$H$ by definition.
So, $K$ has also maximum degree at most~$4$, and thus it is a $(2,4)$-subgraph.
    As the set of chords of~$H$ is a matching, no vertex of degree~$4$ in~$K$ is incident with a chord of~$K$.
    Also, no chord of~$K$ has both endpoints at branching vertices of~$K$.
    Indeed, the chords of~$H$ form a matching, so the degree of an endpoint of a chord of~$K$ does not increase when passing from~$H$ to~$K$.
Thus, two branching endpoints of a chord of~$K$ would also be branching endpoints of the same chord in~$H$, a contradiction.
Thus, $K$ is a skeleton.

\medskip

(2) We show that 
$F_0(K)\cup F_1(K)
    \subseteq
    \big(F_0(H)\cup F_1(H)\big)
    \setminus \{e_1,\ldots,e_m\}$.
This would imply that $\abs{F_0(K)}+\abs{F_1(K)}<\abs{F_0(H)}+\abs{F_1(H)}$.

We claim that if a chord $xy\in F_2(H)$ remains a chord of~$K$,
then $xy\in F_2(K)$.
Let~$P$ be the corridor of~$H$ containing~$x$ and~$y$ as internal
vertices.

Since $xy$ remains a chord of~$K$, the path $xPy$ is contained in
$K$.
Indeed, if $P\neq R$, no vertex of $\internal(xPy)$ is deleted
in the construction of~$K$.
If $P=R$, then both~$x$ and~$y$ lie in
$\internal(v_2Rv_3)$, as all other internal vertices of~$R$ are
deleted and neither~$x$ nor~$y$ can be $v_2$ or $v_3$ because the
chords of~$H$ form a matching.

We next show that every vertex of $\internal(xPy)$ still has degree
$2$ in~$K$.
Suppose that some $z\in\internal(xPy)$ is incident in
$K$ with an edge not belonging to~$H$.
If this edge is one of $u_2v_2$ and $u_3v_3$, then the corresponding
chord of~$H$ has exactly one endpoint in $V(xPy)$, contradicting
$xy\in F_2(H)$.

Otherwise, this edge is a parity-preserving chord of $H'$ that was
absorbed into~$K$.
Such a chord cannot have both endpoints in
$V(xPy)$, since $xPy$ is contained in one component of $H'$ and,
by Claim~\ref{claim:pbchord}, every chord of $H'$ with both endpoints
in the same component is parity-breaking.
Thus, this chord again
has exactly one endpoint in $V(xPy)$, contradicting
$xy\in F_2(H)$.

Thus, every vertex of $\internal(xPy)$ has degree~$2$ in~$K$.
Moreover, since the chords of~$H$ form a matching, neither~$x$ nor
$y$ is incident with an edge absorbed into~$K$.
Hence~$x$ and~$y$
are internal vertices of the same corridor of~$K$ containing
$xPy$.

Finally, if there is a chord of~$K$ with exactly one endpoint in
$V(xPy)$, then, as every chord of~$K$ is also a chord of~$H$, this
would contradict $xy\in F_2(H)$.
Therefore,
$xy\in F_2(K)$,
proving the claim.

It follows that 
 $F_0(K)\cup F_1(K)
    \subseteq
    F_0(H)\cup F_1(H)$.

On the other hand, none of $e_1,\ldots,e_m$ is a chord of~$K$.
Indeed, $e_2$ and $e_3$ are absorbed into~$K$ as edges, while for
$i=1$ or $i\ge 4$, the vertex $v_i$ is deleted in the construction
of~$K$.
Since
$e_1,\ldots,e_m\in F_0(H)$,
we obtain
$F_0(K)\cup F_1(K)
    \subseteq
    \big(F_0(H)\cup F_1(H)\big)
    \setminus \{e_1,\ldots,e_m\}$.
Consequently,
\[\abs{F_0(K)}+\abs{F_1(K)}\le \abs{F_0(H)}+\abs{F_1(H)}-m<\abs{F_0(H)}+\abs{F_1(H)}\]
as $m\ge 4$.
\end{proof}

\begin{lemma}\label{lem:recursion}
    Let $f(k)=208k\log k$.
Then, for every $k\ge 2$,
    $f(k)\ge f(k-1)+54\log\frac{f(k)}{2}$.
\end{lemma}
\begin{proof}
    Let $x=\log k$.
Since
    $\left(1+\frac{1}{k-1}\right)^{k-1}\ge 2$,
    we have
    \begin{align*}
        f(k)-f(k-1)
        =208\left(
            \log k
            +(k-1)\log\frac{k}{k-1}
        \right)
        \ge 208(x+1).
    \end{align*}
    Moreover, since $\ln x\le x-1$ for $x\ge1$,
    $\log x\le\frac{x-1}{\ln 2}$.
    Hence,
    \begin{align*}
        &208(x+1)-54\log\frac{f(k)}{2}\\
        &\quad \ge 208(x+1)-54 \left(\log \frac{208}{2}+x + \frac{x -1}{\ln2} \right)  \\
        &\quad\ge
        \bigl(416-54\log 208\bigr)
        +(x-1)\left(154-\frac{54}{\ln2}\right)
        \ge0.
    \end{align*}
    This proves the lemma.
\end{proof}

\begin{lemma}\label{lem:manybranching}
     Let~$G$ be a graph and let~$H$ be a skeleton in~$G$. 
    If $\epsilon(H)> 208k\log k$, then~$G$ contains an induced packing of~$k$ even cycles. 
\end{lemma}
\begin{proof}
    Let $f(k)=208k\log k$ and $f_1(k)=104k\log k$.
    
    We prove the statement by induction on    $k+|F_0(H)\cup F_1(H)|$.
    If $k=1$, then it is immediate. 
    So, we assume that $k>1$.

    We first consider the case when every corridor of~$H$ meets at most~$3$ chords in $F_0(H)$. 
    Since every vertex of~$H$ has degree at most~$4$ in~$H$, we have $\epsilon(H)\le 2|V_{\ge 3}(H)|$ and \[|V_{\ge 3}(H)|\ge \epsilon(H)/2 > f(k)/2=f_1(k).\]    
        By applying Lemma~\ref{lem:girth} to the graph obtained from~$H$ by suppressing all degree-$2$ vertices, $H$ contains a cycle~$C$ having at most $2\log f_1(k)$ vertices of $V_{\ge 3}(H)$.

        Note that each branching vertex of~$H$ contained in~$C$ has at most two neighbors in $N_G(V(C))\cap V(H)$.
Moreover, since the set of chords forms a matching, each vertex of degree~$2$ in~$H$ contained in~$C$ has at most one neighbor in  $N_G(V(C))\cap V(H)$.
        As every corridor of~$H$ meets at most~$3$ chords in $F_0(H)$, 
        \[\abs{N_G(V(C))\cap V(H)}\le 2\cdot 2\log f_1(k) + 1\cdot 3\cdot 2\log f_1(k)=10\log f_1(k).\]
        Since every vertex in $N_G(V(C))\cap V(H)$ has degree at most~$4$ in $G[V(H)]$ and has a neighbor in $V(C)$, it has at most three neighbors in $V(H)\setminus V(C)$.
After deleting such a vertex from $H-V(C)$, the recursive deletion of degree-$1$ vertices propagates along at most three corridors, and hence can cause at most three additional branching vertices to cease being branching vertices.
        
        Let $H'$ be the graph obtained from $H-V(C)-N_G(V(C))\cap V(H)$ by repeatedly deleting vertices of degree at most~$1$. 
        Note that deleting a vertex in $V(C)\cup (N_G(V(C))\cap V(H))$ can decrease $\epsilon(H)$ by at most~$2$ through the contribution of the deleted vertex itself.
        By Lemma~\ref{lem:recursion}, we have 
        \begin{align*}
            \epsilon(H')&\ge \epsilon(H) -2(2\log f_1(k) + 10\log f_1(k))-30\log f_1(k) \\
            &= \epsilon(H) -54\log f_1(k) \\
            &> f(k)-54\log f_1(k)\ge f(k-1). 
        \end{align*}

     Since $H'$ is a $(2,4)$-subgraph of~$H$, $H'$ is a semi-skeleton in~$G$.
Let $H''$ be a skeleton completion of $H'$.
As $H''$ is obtained from $H'$ by adding some edges, $\epsilon(H'')\ge \epsilon(H')$.

    We claim that $F_0(H'')\cup F_1(H'')\subseteq F_0(H)\cup F_1(H)$, which allows us to apply the induction hypothesis to $H''$.

    \begin{claim}\label{claim:containment}
        $F_0(H'')\cup F_1(H'')\subseteq F_0(H)\cup F_1(H)$.
    \end{claim}
    \begin{clproof}
        Let $e=xy\in F_0(H'')\cup F_1(H'')$.
Since $H''$ is obtained from $H'$ by adding some chords of $H'$
as edges, every chord of $H''$ is also a chord of $H'$, and hence
a chord of~$H$.

Suppose for a contradiction that $e\notin F_0(H)\cup F_1(H)$.
Then $e\in F_2(H)$.
Let~$P$ be the corridor of~$H$ containing
$x$ and~$y$ as internal vertices.
By the definition of $F_2(H)$,
there is no chord of~$H$ having exactly one endpoint in $V(xPy)$.

Since both~$x$ and~$y$ remain in $H'$, no vertex of
$\internal(xPy)$ is deleted in the construction of $H'$;
otherwise, the subsequent deletion of degree-$1$ vertices along
$P$ would delete at least one of~$x$ and~$y$.
Thus, $xPy$ is contained in $H'$.

Moreover, no edge added in obtaining the skeleton completion $H''$
has exactly one endpoint in $V(xPy)$, since every such added edge
is a chord of~$H$.
Consequently, no internal vertex of $xPy$ becomes a branching
vertex when transforming from $H'$ to $H''$.
Since the chords of~$H$ form a matching, neither~$x$ nor~$y$ is
incident with an added edge.
Thus, $x$ and~$y$ remain internal vertices of the same corridor
of $H''$.

Finally, every chord of $H''$ is a chord of~$H$, so there is no
chord of $H''$ having exactly one endpoint in $V(xPy)$.
Therefore, $e\in F_2(H'')$, a contradiction.
This proves the claim.
    \end{clproof}
    
    By Claim~\ref{claim:containment}, 
    we have $\abs{F_0(H'')\cup F_1(H'')}\le |F_0(H)\cup F_1(H)|$.
    Since $\epsilon(H'')>f(k-1)$, by the induction hypothesis, $H''$ contains an induced packing of $k-1$ even cycles. 
    Since $H'$ is obtained from~$H$ by deleting $V(C)$ and all vertices of~$H$ adjacent to~$C$ in~$G$, there is no edge of~$G$ between $V(C)$ and $V(H')$.
As $H''$ has the same vertex set as $H'$, there is no edge of~$G$ between $V(C)$ and $V(H'')$ either.
    Therefore, together with~$C$, $G$ contains an induced packing of~$k$ even cycles.

    \medskip
    Now, we consider the case when there is a corridor~$R$ meeting at least
$4$ chords in $F_0(H)$.
In this case, we consider $\phi(H,R)$.
Let $K=\phi(H,R)$.
By (1) of Lemma~\ref{lem:phi-bipartite}, $K$ is a skeleton.
Assume that vertices and edges are labeled as in the construction of~$K$.

We observe that $\epsilon(K)\ge \epsilon(H)$.
In the construction of $H'$, deleting the internal vertices of
$R$ outside $v_2Rv_3$ decreases the degrees of the two ends of~$R$ by
one, and hence decreases $\epsilon$ by at most~$2$ in total.
On the other hand, adding $u_2v_2$ and $u_3v_3$ increases the degrees
of $u_2$ and $u_3$ by one, respectively.
Since neither $u_2$ nor $u_3$
has degree~$4$ in~$H$, this increases $\epsilon$ by~$2$ in total.
Moreover, adding parity-preserving chords in the construction of
$K$ can only increase $\epsilon$.
Thus,
$\epsilon(K)\ge \epsilon(H)>f(k)$.

Moreover, by (2) of Lemma~\ref{lem:phi-bipartite},
$\abs{F_0(K)\cup F_1(K)}
    <
    \abs{F_0(H)\cup F_1(H)}$.
Hence, the induction parameter $k+|F_0(K)\cup F_1(K)|$
is strictly smaller than
$k+|F_0(H)\cup F_1(H)|$.
Therefore, by the induction hypothesis applied to~$K$,
the graph~$G$ contains an induced packing of~$k$ even cycles.
\end{proof}

\section{Path-matching graphs}\label{sec:pathmatching}

In this section, we study path-matching graphs, consisting of a collection of vertex-disjoint paths together with a matching on their vertices.
Our main goal is to show that, under a mild restriction on the arrangement of the matching edges, sufficiently many matching edges guarantee many vertex-disjoint cycles containing an even number of matching edges.
This will be used later to control the chords arising in a bipartite ear decomposition.

A graph~$G$ is a \emph{path-matching} graph if $G=(P_1\cup\cdots\cup P_m) + N$ where $P_1,\dots,P_m$ are vertex-disjoint paths and~$N$ is a matching on $V(G)$ such that
$N\cap \bigcup_{i\in[m]}E(P_i)=\emptyset$.
For each $i\in [m]$, we fix a linear ordering $<_{P_i}$ of $V(P_i)$ induced by traversing $P_i$ from one endpoint to the
other. 

The edges in~$N$ will represent parity-breaking paths of length
at least~$2$, rather than actual edges of the underlying graph.
For this reason, we refer to them as \emph{links}.
A cycle of~$G$ is an \emph{even-$N$ cycle} if it contains an even number of links in~$N$. 

A link $uv\in N$ is called a \emph{same-path link} if
$u,v\in V(P_i)$ for some $i\in [m]$.
For such a link $uv$, we denote by $I_{uv}$ the subpath $uP_iv$.

Two distinct same-path links $e=ab$ and $f=cd$ in~$N$ are called
\emph{serial} if their endpoints lie on the same path $P_i$ and,
after possibly exchanging~$e$ and~$f$ and their endpoints,
$a<_{P_i} b<_{P_i} c<_{P_i} d$.
Moreover, two serial links $e=ab$ and $f=cd$, ordered as above, are called \emph{consecutive serial} if
$(V(aP_i d)\setminus \{a,b,c,d\})\cap V(N)=\emptyset$,
where $V(N)$ denotes the set of all endpoints of links in~$N$.
In other words, $a,b,c,d$ appear consecutively in $V(N)\cap V(P_i)$ with respect to $<_{P_i}$.
We say that~$G$ satisfies the \emph{consecutive serial property} if it contains no pair of consecutive serial links.

A vertex set $Y\subseteq V(G)$ is said to be \emph{closed under~$N$} if for every $uv\in N$, $u\in Y$ implies $v\in Y$.
We will need a hitting set for all even-$N$ cycles that is closed under~$N$.

We will use the ordinary \ep property of $\theta$-minor models.
The $\theta$ graph is the multigraph consisting of two vertices with three parallel edges between them.
\begin{theorem}[Chatzidimitriou,  Raymond, Sau, and Thilikos~\cite{ChatRST2018}]\label{thm:thetaminor}
    There is a function $g_{\theta}(k)=O(k\log k)$ satisfying the following.
For every graph~$G$ and integer~$k$, $G$ contains either~$k$ vertex-disjoint $\theta$-minor models, or a vertex set~$X$ with $\abs{X}\le g_{\theta}(k)$ such that $G-X$ has no $\theta$-minor models.
Moreover, one of these two outcomes can be found in polynomial time.
\end{theorem}
Note that the degree of the polynomial in the running time of Theorem~\ref{thm:thetaminor} depends on the algorithm of Bazgan, Tuza, and Vanderpooten~\cite{BazganTV2007} for partitioning a graph into subgraphs with prescribed minimum degree.
Although they did not explicitly state the running time, a straightforward implementation of their algorithm shows that the running time in Theorem~\ref{thm:thetaminor} can be bounded by $O(n^5)$.
Cames van Batenburg, Huynh, Joret, and Raymond~\cite{CamesHJR2019} generalized this result to planar~$H$-minors, obtaining an $O_H(k\log k)$ Erdős--Pósa bound together with a polynomial-time algorithm for finding one of the two outcomes.

We prove the following. 

\begin{lemma}\label{lem:counting-theta-free-virtual-links}
Let $J=(P_1\cup\cdots\cup P_m) + N$ be a path-matching graph satisfying the consecutive serial property and let $F=P_1\cup\cdots\cup P_m$.
Let $Y\subseteq V(J)$ be closed under~$N$ such that $J-Y$ has no $\theta$-minor model.
Let~$q$ be the number of components of $F-Y$, and $N_Y=\{uv\in N:u,v\notin Y\}$.
Then $\abs{N_Y}\le 11q$.
\end{lemma}
\begin{proof}
Let $J_Y=J-Y$, $F_Y=F-Y$, and let~$K$ be the multigraph obtained from $J_Y$ by contracting every component of $F_Y$ to one vertex.
Thus, every edge of $N_Y$ becomes either
a loop or a non-loop edge of~$K$.
Let
$\ell$ be the the number of loops of~$K$, and $r=|N_Y|-\ell$.

Let $K_0$ be obtained from~$K$ by deleting all loops, and let
$c$ be the number of components of $K_0$.
Since $K_0$ is a minor of $J-Y$, it has no
$\theta$-minor model.
Thus, every block of $K_0$ is either an isolated vertex, a single
edge, or a cycle.
In other words, $K_0$ is a cactus
multigraph.

For every connected component~$C$ of $K_0$, let $b(C)$ denote
the number of cycle blocks of~$C$.
Since~$C$ is a cactus,
$\abs{E(C)}=|V(C)|-1+b(C)$. 
As $b(C)\le |V(C)|-1$.
we have 
$\abs{E(C)}\le 2(|V(C)|-1)$.
Summing over the components of $K_0$, we obtain
$r=|E(K_0)|\le 2(q-c)$.

We next bound the number $\ell$ of loops.
Let~$R$ be a component of $F_Y$, and let $v_R$ be the
corresponding vertex of~$K$. 
We use $<_{R}$ for the ordering
of $V(R)$ inherited from the path of $\{P_i:i\in [m]\}$ containing~$R$.

We first observe that the intervals of any two distinct links in
$N_Y$ whose endpoints are both contained in~$R$ are disjoint.
Indeed, let $e=ab$ and $f=cd$ be two such links, and without loss of generality, we assume that $a<_R b$, $c<_R d$, and $a<_R c$.
Since~$N$ is a matching, the four endpoints are distinct.
If $I_e$ and $I_f$ intersect, then either
$
a<_{R}c<_{R}b<_{R}d$
or
$
a<_{R}c<_{R}d<_{R}b$.
In the first case, $R+\{e,f\}$ contains a subdivision of
$\theta$ with branch vertices~$c$ and~$b$, and in the second
case it contains a subdivision of $\theta$ with branch vertices
$c$ and~$d$.
This contradicts the assumption that $J-Y$ contains no
$\theta$-minor model.

Let $e_1,\ldots,e_t$ be all links in $N_Y$ whose two endpoints
lie in~$R$.
By the observation above, we may write
$e_j=a_jb_j$ and index them so that
\[
a_1<_{R}b_1<_{R}a_2<_{R}b_2<_{R}\cdots
<_{R}a_t<_{R}b_t.
\]
By the consecutive serial condition, for each $i\in[t-1]$, $a_iRb_{i+1}-\{a_i, b_i, a_{i+1}, b_{i+1}\}$ contains an endpoint of some link of~$N$, say $z_i$.
Let $f_i\in N$ be the link incident with $z_i$.
Since $z_i\in V(R)\subseteq V(F)\setminus Y$ and~$Y$ is closed
under~$N$, both endpoints of $f_i$ lie outside~$Y$.
This implies that 
$f_i\in N_Y$.
As $e_1, \ldots, e_t$ are all links in $N_Y$ whose two endpoints lie in~$R$, $f_i\notin \{e_1, \ldots, e_t\}$, and $f_i$ is incident with a component of $F_Y$ other than~$R$.
That is, $f_i$ is a non-loop edge of~$K$.

For each $j\in \left[\lfloor t/2\rfloor\right]$, let $f_j'=f_{2j-1}$.
By the construction, $f_1', \ldots, f_{\lfloor t/2\rfloor}'$ are pairwise distinct.
It follows that
$\lfloor t/2\rfloor \le \deg_{K_0}(v_R)$
and 
$t\le 2\deg_{K_0}(v_R)+1$.
Summing it over all components~$R$ of $F_Y$, we obtain
\[
\ell
\le
2\sum_{v\in V(K_0)}\deg_{K_0}(v)+q
=
4r+q.
\]
Thus, we get
$\abs{N_Y}
=\ell+r\le 5r+q\le 10(q-c)+q=11q-10c\le 11q$,  as desired.
\end{proof}

\begin{proposition}\label{prop:many-links}
Let $J=(P_1\cup\cdots\cup P_m) + N$ be a path-matching graph satisfying the consecutive serial property.
If
$\abs{N}>11m+23g_\theta(k)$,
then~$J$ contains~$k$ pairwise vertex-disjoint even-$N$
cycles.
\end{proposition}
\begin{proof}
Let $F=P_1\cup\cdots\cup P_m$.

Suppose first that~$J$ contains~$k$ pairwise vertex-disjoint
$\theta$-minor models.
Since the $\theta$ graph has maximum degree three, each such model contains
a subdivision of $\theta$.
Let $R_1,R_2,R_3$ be the three
internally vertex-disjoint paths between the two branch vertices of
one such subdivision.
Among $R_1,R_2,R_3$, two contain the same parity of links in~$N$.
The union of these two paths is therefore an even-$N$ cycle.
Choosing one such cycle from each of the~$k$ vertex-disjoint
$\theta$-minor models gives~$k$ pairwise vertex-disjoint even-$N$
cycles.

We may therefore assume that~$J$ contains no~$k$ pairwise
vertex-disjoint $\theta$-minor models.
By Theorem~\ref{thm:thetaminor}, there exists
a set $X\subseteq V(J)$ with $\abs{X}\le g_\theta(k)$
such that $J-X$ has no $\theta$-minor model.

We extend~$X$ to a set closed under~$N$.
Namely, let
\[\overline X=
X\cup
\{v\in V(F):uv\in N\text{ for some }u\in X\}.\]
Since~$N$ is a matching, we have $\abs{\overline X}\le 2|X|$.
Moreover, $\overline X$ is closed under~$N$, and
$J-\overline X$
has no $\theta$-minor model.

Let $N'$ be the set of all links $uv$ in~$N$ such that $\{u,v\}\cap X=\emptyset$.
Since~$N$ is a matching, at most $\abs{X}$ edges of~$N$ meet~$X$.
Thus
we have $\abs{N'}\ge |N|-|X|$.

Let~$q$ be the number of components of
$F-\overline X$.
Deleting one vertex from a disjoint union of paths increases the number of
components by at most one.
Therefore, we have 
$q\le m+|\overline X|
\le m+2|X|$.

Applying Lemma~\ref{lem:counting-theta-free-virtual-links} with
$Y=\overline X$, we obtain
$\abs{N'}\le 11q$.
Combining the previous inequalities, we obtain
$\abs{N}-|X|
\le |N'|
\le 11q
\le 11m+22|X|$.
Therefore,
$\abs{N}
\le 11m+23|X|
\le 11m+23g_\theta(k)$.
This contradicts our assumption.

Thus~$J$ contains~$k$ pairwise vertex-disjoint
$\theta$-minor models, and the first part of the proof yields~$k$
pairwise vertex-disjoint even-$N$ cycles.
\end{proof}

We will use the following definition. 
Let $F=P_1\cup\cdots\cup P_m$ be a disjoint union of paths in a graph~$G$, and
let $\mathcal Q$ be a family of vertex-disjoint~$F$-paths such that their
endpoints are pairwise distinct.
For each $Q\in\mathcal Q$, let
$x_Q$ and $y_Q$ be its endpoints, and let
$N(\mathcal Q):=\{x_Qy_Q:Q\in\mathcal Q\}$.
We define
\[
\aux(F,\mathcal Q)=F+N(\mathcal Q),
\]
and refer to $x_Qy_Q$ as the link corresponding to $Q\in \mathcal{Q}$.
We regard $\aux(H,\mathcal{Q})$ as a path-matching graph, where $N(\mathcal{Q})$ is the set of links.

\section{Reduction, ear decompositions, and the final cover}\label{sec:reduction}

We reduce Theorem~\ref{thm:main} to the following theorem.
Similar reduction arguments have been used in various Erd\H{o}s–P\'osa-type results.

\begin{theorem}\label{thm:main2}
Let~$G$ be a graph, $k$ be a positive integer, and~$C$ be a shortest even cycle of~$G$ such that $G-N_G[C]$ does not contain an even cycle.
Then~$G$ contains an induced packing of~$k$ even cycles or a set~$X$ of at most $O(k\log k)$ vertices in~$G$ such that $G-N_G[X]$ has no even cycle.
Moreover, there is an algorithm running in time polynomial in $\abs{V(G)}$ and~$k$ that outputs one of these two outcomes.
\end{theorem}

We present the proof of Theorem~\ref{thm:main} assuming that Theorem~\ref{thm:main2} holds. 
We use the following algorithm to find a shortest even cycle in a graph. 

\begin{theorem}[Yuster and Zwick~\cite{YusterZ1997}]\label{thm:alg-evencycle}
    A shortest even cycle in a graph~$G$ can be found in time $O(|V(G)|^2)$.
\end{theorem}

When $G_{j+1}$ has a set $X_{j+1}$ whose closed neighborhood meets all even cycles, we apply the induction hypothesis to $G_j-N_{G_{j+1}}[X_{j+1}]$.

\begin{proof}[Proof of Theorem~\ref{thm:main}]
    Let~$c$ be a constant such that the set~$X$ in
Theorem~\ref{thm:main2} can always be chosen with
$\abs{X}\le ck\log k$ for every $k\ge2$.
Let $\mu_k=ck\log k$.
    
    By repeatedly applying Theorem~\ref{thm:alg-evencycle}, we can construct, in polynomial time, a sequence of graphs $G_1,\ldots, G_{\ell+1}$ and a sequence of cycles $C_1,\ldots , C_{\ell}$ in polynomial time  
such that 
\begin{itemize}
\item $G_1=G$, 
\item for each $i\in [\ell]$, $C_i$ is a shortest even cycle of $G_i$,  
\item for each $i\in [\ell]$, $G_{i+1}=G_i-N_{G_i}[C_i]$, and
\item $G_{\ell+1}$ has no even cycles.
\end{itemize}
If $\ell \geq k$, then $\{C_1, \ldots, C_k\}$ is an induced packing of~$k$ even cycles in~$G$.
Thus, we may assume that $\ell \leq k-1$.
In the following claim, we show that either we can find an induced packing of~$k$ even cycles or we can recursively construct a set in $G_j$ whose closed neighborhood meets all even cycles. 

\begin{claim}
For each $j \in [\ell+1]$, one can find in polynomial time either an induced packing of~$k$ even cycles of~$G$, or 
a set $X_{j}$ of at most $(\ell+1-j)\mu_k$ vertices in $G_{j}$ so that $G_j-N_{G_j}[X_j]$ has no even cycles.
\end{claim}

\begin{clproof}
We proceed by reverse induction on~$j$.
When $j=\ell+1$ the statement holds with $X_{\ell+1}=\emptyset$ since $G_{\ell+1}$ has no even cycle.

Fix $j\leq \ell$.
We apply induction on $G_{j+1}$.
If $G_{j+1}$ has an induced packing of~$k$ even cycles, then we are done.
Thus, we may assume that there is a set $X_{j+1}$ of vertices of $G_{j+1}$ of size at most $(\ell+1-(j+1))\mu_k=(\ell-j)\mu_k$   such that $G_{j+1}-N_{G_{j+1}}[X_{j+1}]$ has no even cycles.
Since $G_{j+1}=G_j-N_{G_j}[C_j]$, $C_{j}$ is a shortest even cycle of $G_{j}-N_{G_{j+1}}[X_{j+1}]$, and $\left( G_{j}-N_{G_{j+1}}[X_{j+1}] \right)-N_{G_j}[C_j]$ has no even cycles.

By applying Theorem~\ref{thm:main2} to $(G_{j}-N_{G_{j+1}}[X_{j+1}], k, C_{j})$, 
we can find in polynomial time either an induced packing of~$k$ even cycles of $G_{j}-N_{G_{j+1}}[X_{j+1}]$, or a vertex set~$X$ of size at most $\mu_k$ whose closed neighborhood hits all even cycles of of $G_{j}-N_{G_{j+1}}[X_{j+1}]$. 

In the former case, we output an induced packing of~$k$ even cycles, and we are done.
In the latter case, we set $X_{j}=X_{j+1}\cup X$.
Observe that $G_{j}- N_{G_j}[X_j]$ has no even cycles and $\abs{X_{j}}\le (\ell+1-j)\mu_k$ as claimed. 
\end{clproof}
In particular, either we obtain an induced packing of~$k$ even cycles, or $N_{G}[X_1]$ hits all even cycles of $G_1=G$ and $\abs{X_1}\le \ell\mu_k\le (k-1)ck\log k=O(k^2\log k)$, as desired. 
\end{proof}

In the rest of this section, we focus on proving Theorem~\ref{thm:main2}. 
We assume that~$G$ is a graph, $k$ is a positive integer, and~$C$ is a shortest even cycle of~$G$ such that $G-N_G[C]$ does not contain an even cycle.
If~$C$ has length~$4$, then we can take $X=V(C)$.
Thus, we may assume that $\abs{V(C)}\ge 6$, and~$G$ is $C_4$-subgraph-free.

\subsection{Bipartite ear decompositions}\label{subsec:bipartiteear}
 Based on Proposition~\ref{prop:boundingchord}, we will take a certain ear decomposition of~$G$ starting from the given shortest even cycle~$C$.
Note that~$C$ has at most one chord by Lemma~\ref{lem:shortestevencycle-chords}, as~$C$ has length at least~$6$.

A \emph{bipartite ear decomposition} in~$G$ is a sequence $\{(P_i, M_i, J_i, T_i, S_i) \}_{i\in [m]}$ for some integer~$m$ defined as follows.
\begin{itemize}
    \item Let $P_1=C$ and $S_1=T_1$ be the set of endpoints of chords of~$C$.
Let $M_1$ be the set of chords of~$C$, $J_1$ be the graph $(V(P_1),M_1)$.
    \item Let~$i$ be a positive integer.
Let $H_i=\bigcup_{j\in [i]}P_j$ and $\mathcal{S}_i=\bigcup_{j\in [i]}S_j$.
    If $G-N_{G-(V(H_i)\setminus \mathcal{S}_i)}[\mathcal{S}_i]$ has no $H_i$-extension, then stop the procedure.
Otherwise, we define the following.
    \begin{itemize}
        \item Let $P_{i+1}$ be a shortest path among all $H_i$-extensions contained in $G-N_{G-(V(H_i)\setminus \mathcal{S}_i)}[\mathcal{S}_i]$.
        \item Let $M_{i+1}$ be the set of chords of $H_i\cup P_{i+1}$ in~$G$ having an endpoint in $\internal(P_{i+1})$.
        \item Let $J_{i+1}=(V(H_i\cup P_{i+1}), M_{i+1})$.
        \item Let $T_{i+1}$ be a vertex cover of $J_{i+1}$ given by Proposition~\ref{prop:boundingchord} applied to $(G,C,H_i,\mathcal{S}_i,P_{i+1})$.
        \item Let $S_{i+1}$ be the union of $T_{i+1}$ and the set of vertices of $H_i\cup P_{i+1}$ that are at distance at most~$2$ from an endpoint of $P_{i+1}$.
    \end{itemize}
\end{itemize}
We call~$i$ a \emph{non-matching index} if $V_{\ge 2}(J_i)\neq\emptyset$.
Note that if $V_{\ge 2}(J_i)=\emptyset$, then $J_i$ is a matching.
If~$i$ is a non-matching index and $v\in V_{\ge 2}(J_i)$, then by (1) of Proposition~\ref{prop:boundingchord}, there are two chords~$vz_1$ and~$vz_2$ incident with~$v$ such that $z_1, z_2\in V(P_i)$ and $z_1P_iz_2\cup z_2vz_1$ is an even cycle.

We say that a bipartite ear decomposition $\{(P_i, M_i, J_i, T_i, S_i) \}_{i\in [m]}$ is \emph{maximal} if there is no bipartite ear decomposition $\{(P_i', M_i', J_i', T_i', S_i') \}_{i\ge [m']}$ such that $m< m'$ and for all $i\in [m]$, $P_i=P_i'$.
In the remainder of this section,  we fix a maximal bipartite ear decomposition $\{(P_i, M_i, J_i, T_i, S_i) \}_{i\in [m]}$.

We include in $S_i$ all vertices of $H_i$ within distance~$2$ of the endpoints of $P_i$ to ensure that, in the subsequent step, every path in $\{P_i\}_{i\in [m]\setminus \{1\}}$ has length at least~$3$. 
We prove this in the next lemma.

\begin{lemma}\label{lem:length2Hpath}
     Each path in $\{P_i\}_{i\in [m]\setminus \{1\}}$ has length at least~$3$.
\end{lemma}
\begin{proof}
    Suppose that $P_i$ has length at most~$2$ for some $i\in [m]\setminus \{1\}$.
Let~$v$ and~$w$ be the endpoints of $P_i$. 
    For $z\in \{v,w\}$, let $\tau(z)=\min\{r\in [i-1]:z\in V(H_r)\}$.
In other words, for $z\notin V(C)$, $\tau(z)$ is the minimum integer~$j$ such that~$z$ is an internal vertex of $P_j$. 

    Suppose first that $P_i$ has length~$1$.
Then $vw$ is a chord of $H_{i-1}$.
Let $j=\max\{\tau(v),\tau(w)\}$.
    If $j=1$, then $vw$ is a chord of~$C$, and hence
    $v,w\in S_1\subseteq\mathcal S_{i-1}$.
    If $j\ge2$, then 
    one of~$v$ and~$w$ is an internal vertex of $P_j$, and $vw\in M_j$.
    Since $T_j$ is a vertex cover of $J_j$, one of~$v$ and~$w$ belongs to 
    $T_j\subseteq S_j\subseteq\mathcal S_{i-1}$.
    In either case, this contradicts the choice of $P_i$.

     Suppose now that $P_i$ has length~$2$ for some $i\in [m]\setminus\{1\}$. 
    We divide cases depending on whether $\tau(v)=\tau(w)$.

    \medskip \noindent
    \textbf{Case 1. $\tau(v)=\tau(w)=j$ for some $j\in [m]$.}

First assume that $j=1$.
Then $v,w\in V(C)$.
Since $H_m$ is bipartite and
$P_i$ has length~$2$, the two $(v,w)$-paths in~$C$ have even
length.
Neither of them has length~$2$, since otherwise $C\cup P_i$ contains a subgraph isomorphic to $C_4$.
Thus, both have
length at least~$4$.
Consequently, the union of $P_i$ and
a $(v,w)$-path in~$C$ is an even cycle shorter than~$C$,
a contradiction.

   Suppose then that $j\ge 2$.
In this case,
$v,w\in V(\internal(P_j))$.
Since $H_m$ is bipartite,
$vP_jw$ has even length.
If $\abs{E(vP_jw)} = 2$, then $vP_jw \cup P_i$ is isomorphic to $C_4$, contradicting the assumption that~$G$ is $C_4$-subgraph-free.
Thus, $\abs{E(vP_jw)} > 2$, and in this case, the path obtained from $P_j$ by replacing $vP_jw$ with $P_i$ is a shorter $H_{j-1}$-extension than $P_j$.
That is also a contradiction. 
    
    \medskip \noindent
    \textbf{Case 2. $\tau(v)\neq \tau(w)$.}

Without loss
of generality, let $a=\tau(v)<\tau(w)=b.$
Then
$v\in V(H_{b-1})$ and $w\in V(\internal(P_b))$.
Let~$x$ and~$y$ be the endpoints of $P_b$, where we choose~$x$ so
that $x\ne v$.

Since~$w$ is an endpoint of $P_i$, we have
$w\notin\mathcal S_{i-1}$, and $w\notin S_b$.
By the definition of $S_b$,
$\dist_{H_b}(w,\{x,y\})\ge 3.$
    
    Since $\dist_{P_b}(w, \{x,y\})\ge 3$, $xP_bw\cup P_i$ is a shorter $H_{j-1}$-path than $P_j$.
    This new path preserves the bipartite condition, as $H_m$ is bipartite.
This contradicts the choice  of $P_j$.

    We conclude that $P_i$ has length at least~$3$.
\end{proof}

\begin{lemma}\label{lem:atmostonechord}
    For each $i\in [m]$, every vertex of $H_i-\mathcal{S}_i$ is incident with at most one chord of $H_i$.
\end{lemma}
\begin{proof}
Suppose for a contradiction that some vertex
$v\in V(H_i)\setminus\mathcal S_i$ is incident with two
distinct chords of $H_i$.
Let $vy$ be a chord of $H_i$ with $vy\in M_b$ such that~$b$ is maximum. 

If $v\in V(\internal(P_b))$, then~$v$ is incident with two chords in $H_b$, which implies that $v\in T_b$.
However,  $v\notin T_b\subseteq \mathcal{S}_i$, which is a contradiction.
Thus, $v$ is contained in $H_{b-1}$ and $y\in V(\internal(P_b))$.
Moreover, as $v\notin S_b$ and $S_b$ contains the endpoints
of $P_b$, we have
$v\in V(H_{b-1})\setminus V(P_b)$.

We claim that $\deg_{J_b}(y)=1$.
Suppose that $yz\in M_b$ for
some $z\ne v$.
If $z\in V(P_b)$, then the chords $vy$ and
$yz$ contradict Lemma~\ref{lem:chordandedge}, since~$y$ is
not an internal vertex of $yP_bz$.
If
$z\in V(H_{b-1})\setminus V(P_b)$,
then $vy$ and $zy$ contradict
Lemma~\ref{lem:shorthestHpath-edges-2}.
Therefore, $\deg_{J_b}(y)=1$.

Since $v\notin T_b$ and
$V_{\ge2}(J_b)\subseteq T_b$, while $vy\in M_b$, we also have
$\deg_{J_b}(v)=1$.
Thus
\[
v,y\in Z_b\subseteq T_b\subseteq S_b\subseteq\mathcal S_i,
\]
contrary to $v\notin\mathcal S_i$.
\end{proof}

We show that if there are $11k$ many non-matching indices, then we can find an induced packing of~$k$ even cycles. 

\begin{lemma}\label{lem:nonmatching}
    If there are $11k$ non-matching indices, then~$G$ contains an induced packing of~$k$ even cycles.
\end{lemma}
\begin{proof}
    Let~$I$ be the set of non-matching indices, and assume that $\abs{I}\ge 11k$.
For each $i\in I$, let $w_i\in V_{\ge 2}(J_i)$.
By Proposition~\ref{prop:boundingchord}, there exist two distinct
vertices $z_i^1,z_i^2\in V(P_i)\setminus\{w_i\}$ such that
$w_i z_i^1,w_i z_i^2\in M_i$
and
$w_i\notin V(z_i^1P_i z_i^2).$
    Let
$D_i=w_i z_i^1\cup z_i^1P_i z_i^2\cup z_i^2w_i$.
    Note that $D_i$ is an even cycle.

We observe the following property.
\begin{claim}\label{claim:nochordDi}
    Let $i\in I$.
Then $\abs{V(D_i)\cap V(H_{i-1})}\le 1$.
    Also, there is no chord of $H_m$ between $V(D_i)\setminus V(H_{i-1})$ and $V(H_{i-1})\setminus V(D_i)$.
\end{claim}
\begin{clproof}
    We first deal with the case when $w_i\in V(H_{i-1})$.
    Then
$z_i^1,z_i^2\in V(\internal(P_i))$.
So, $\abs{V(D_i)\cap V(H_{i-1})}= 1$.

    Now, assume that there is a chord $ab$ with $a\in V(H_{i-1})\setminus V(D_i)$ and $b\in V(D_i)\setminus V(H_{i-1})$.
    If both~$a$ and $w_i$ are endpoints of $P_i$, then there are two chords contradicting Lemma~\ref{lem:shortestHpath-chordsofpath}.
    If exactly one of~$a$ and $w_i$ is an endpoint of $P_i$, then $ab$ and one of $w_iz_i^1$ and $w_iz_i^2$ contradict Lemma~\ref{lem:chordandedge}. 
    Otherwise, neither~$a$ nor $w_i$ is an endpoint of $P_i$.
    So, $w_i\in V(H_{i-1})\setminus V(P_i)$ and
    $a\in V(H_{i-1})\setminus
\bigl(V(P_i)\cup\{w_i\}\bigr).$
Since $w_i$ has two chords $w_i z_i^1$ and $w_i z_i^2$
to $P_i$, whereas~$a$ has the chord $ab$ to $P_i$,
this contradicts Lemma~\ref{lem:shorthestHpath-edges-3}.

Suppose now that $w_i\in V(\internal(P_i))$.
Since
$w_i\notin V(z_i^1P_i z_i^2)$,
the path $z_i^1P_i z_i^2$ is contained in one component of
$P_i-w_i$, and hence contains at most one endpoint of $P_i$.
Thus, we have $\abs{V(D_i)\cap V(H_{i-1})}\le 1$.

Assume that there is a chord $ab$ with $a\in V(H_{i-1})\setminus V(D_i)$ and $b\in V(D_i)\setminus V(H_{i-1})$.
If~$a$ is an endpoint of $P_i$, then there are two chords of $P_i$ violating (2) of Lemma~\ref{lem:shortestHpath-chordsofpath}.
On the other hand, if $a\in V(H_{i-1})\setminus V(P_i)$, then there are two chords of $P_i$ violating Lemma~\ref{lem:chordandedge}.
This is a contradiction.
\end{clproof}

\begin{claim}\label{claim:intersectionDi}
    Let $i,j\in I$ with $i<j$.
If $v\in V(D_i)\cap V(D_j)$, then $v\in S_j$.
\end{claim}
\begin{clproof}
    Assume that $v\in V(D_i)\cap V(D_j)$.
Then $v\in V(D_j)\cap V(H_{j-1})$. 
    If $w_j\in V(H_{j-1})$, then 
    $v=w_j$ and~$v$ is a vertex of degree at least~$2$ in $J_j$.
In this case, $v\in S_j$.
Otherwise, $w_j\in V(\internal(P_j))$ and~$v$ is an endpoint of $P_j$.
Thus, we have $v\in S_j$. 
\end{clproof}
\begin{claim}\label{claim:nothreeintersection}
   Every vertex of $H_m$ is contained in at most two cycles of $\{D_i:i\in I\}$. 
\end{claim}
\begin{clproof}
    Suppose that there is a vertex~$v$ contained in three cycles in $\{D_i:i\in I\}$, say $D_a, D_b, D_c$ with $a<b<c$.
    By Claim~\ref{claim:intersectionDi}, $v$ is contained in $S_b$.
As $b<c$, $v$ has no neighbors in $V(P_c)\setminus V(H_{c-1})$ and~$v$ is not an endpoint of $P_c$. 
    
    On the other hand, $v\in V(D_c)\cap V(H_{c-1})$.
By the definition of $D_c$, $v$ is either an endpoint of $P_c$ or a vertex having a chord to $V(\internal(P_c))$.
Both cases contradict the preceding observation.
\end{clproof}

    Let $\Gamma$ be an auxiliary graph on vertex~$I$ such that for distinct $a,b\in I$, $ab\in E(\Gamma)$ if and only if \begin{itemize}
        \item $V(D_a)\cap V(D_b)\neq \emptyset$ or
        \item there is an edge in~$G$ between $D_a$ and $D_b$.
    \end{itemize} 

    \begin{claim}\label{claim:degreefive}
    For every $i\in I$, $\abs{N_G(D_i)\cap V(H_{i-1})}\le 5$.
    \end{claim}
    \begin{clproof}
        First assume that $w_i\in V(H_{i-1})$.
By Claim~\ref{claim:nochordDi},
$V(D_i)\cap V(H_{i-1})=\{w_i\}$,
and there is no chord between
$V(D_i)\setminus V(H_{i-1})$ and
$V(H_{i-1})\setminus V(D_i)$.

Since $w_i\notin\mathcal S_{i-1}$, by
Lemma~\ref{lem:atmostonechord}, $w_i$ is incident with at most
one chord of $H_{i-1}$.
Moreover, $w_i$ has at most two
neighbors in $H_{i-1}$ by edges of $H_{i-1}$.
The only other possible edges from $D_i$ to $H_{i-1}$ are
the edges of $P_i$ leaving $z_i^1P_i z_i^2$ at its two ends.
There are at most two such edges.
Therefore,
\[
\bigl|N_G(D_i)\cap
V(H_{i-1})\bigr|
\le 2+1+2=5.
\]

Now assume that $w_i\notin V(H_{i-1})$.
Then
$w_i\in V(\internal(P_i))$.
By Claim~\ref{claim:nochordDi},
$\abs{V(D_i)\cap V(H_{i-1})}\le 1$,
and there is no chord between
$V(D_i)\setminus V(H_{i-1})$
and 
$V(H_{i-1})\setminus V(D_i)$.

Suppose first that
$V(D_i)\cap V(H_{i-1})=\{x\}$.
Then~$x$ is an endpoint of $P_i$.
Since $P_i$ is a path in $G-N_{G-(V(H_{i-1})\setminus \mathcal{S}_{i-1})}[\mathcal{S}_{i-1}]$, we have
$x\notin\mathcal S_{i-1}$.
Thus, by
Lemma~\ref{lem:atmostonechord}, $x$ is incident with at most
one chord of $H_{i-1}$.
Moreover, $x$ is incident with at most
two edges of $H_{i-1}$.
Thus, at most three vertices of
$V(H_{i-1})\setminus V(D_i)$ can be adjacent to~$x$.

By Claim~\ref{claim:nochordDi}, every other edge from
$D_i$ to $H_{i-1}$ must be an edge of $P_i$.
As~$x$ is an endpoint of $P_i$, there is at most one edge of $P_i$
joining a vertex of $D_i\setminus V(H_{i-1})$ to a vertex of
$H_{i-1}\setminus V(D_i)$.
Thus, we have 
\[
\bigl|N_G(D_i)\cap
V(H_{i-1})\bigr|
\le 3+1=4.
\]

If $V(D_i)\cap V(H_{i-1})=\emptyset$, then, again by
Claim~\ref{claim:nochordDi}, every edge from $D_i$ to
$H_{i-1}$ is an edge of $P_i$.
There are at most two such
edges, and hence the desired bound follows as well.
    \end{clproof}

By Claims~\ref{claim:intersectionDi}, \ref{claim:nothreeintersection}, and~\ref{claim:degreefive}, for every $i\in I$, there are at most ten indices $a\in I$ with $a<i$ such that $ai\in E(\Gamma)$.
So, $\Gamma$ is $10$-degenerate, and it has an independent set $I'$ of size at least $11k/11=k$.
Then $\{D_i:i\in I'\}$ is an induced packing of~$k$ even cycles.
\end{proof} 

    We now consider a case when~$m$ is large.
    In this case, we show that~$G$ has an induced packing of~$k$ even cycles, using Lemma~\ref{lem:manybranching}. 

\begin{lemma}\label{lem:transformtoskeleton}
    Let~$I$ be the set of all non-matching indices.
Let~$H$ be the graph obtained from $H_m$ by removing $\bigcup_{i\in I}T_i$, and then recursively deleting vertices of degree at most~$1$. 
    Then~$H$ is a semi-skeleton with at least $2(m-1)-24\abs{I}$ branching vertices.
\end{lemma}
\begin{proof}
    By construction, $H$ is bipartite.
By Lemma~\ref{lem:atmostonechord}, every vertex of $H-\mathcal{S}_m$ is incident with at most one chord of~$H$.
    
    Since each subsequent step of the construction is performed outside the neighborhoods of the existing branching vertices, no newly created branching vertex can be joined by a chord to an existing branching vertex.
Thus, no chord joins two branching vertices.

    However, a chord may join an endpoint of a newly added extension path, say $P_j$, to a vertex of the existing construction outside $\mathcal{S}_{j-1}$.
We claim that every endpoint of $P_j$ is incident with at most one such chord.
Suppose not.
Let $va$ and $vb$ be two such chords, where~$v$ is an endpoint of $P_j$.
Before adding $P_j$, the vertex~$v$ belongs to the existing construction outside $\mathcal{S}_{j-1}$, contradicting Lemma~\ref{lem:atmostonechord}.
Therefore, the claim holds.
    
    Consequently, any vertex of degree~$4$ in~$H$ created by the operation $\phi$ is not incident with a chord of~$H$.
Moreover, the chords of~$H$ form a matching in~$G$.
    
    Since all vertices of degree at most~$1$ have been recursively deleted, $H$ is a $(2,4)$-subgraph.
    Thus, $H$ is a semi-skeleton.

    Finally, we count the number of branching vertices of~$H$. 
    Each addition of an extension path increases the number of branching vertices by~$2$. 
    For a graph~$J$, let $\operatorname{core}_2(J)$ denote the graph obtained from~$J$ by recursively deleting vertices of degree at most~$1$.
Let $H_m$ denote the construction before deleting $\bigcup_{i\in I} T_i$ to obtain~$H$.
That is, $H= \operatorname{core}_2(H_m - \bigcup_{i\in I} T_i)$.
    
    We know that $H_m$ has $2(m-1)$ branching vertices.
    Let~$v$ be an arbitrary vertex of $H_m$.
The vertex~$v$ is contained in at most $\deg_{H_m}(v)$ corridors.
    If~$v$ is not a branching vertex of $H_m$, then $\operatorname{core}_2(H_m-v)$ has at least $\abs{V_{\geq 3}(H_m)}-2$ branching vertices.
If~$v$ is a branching vertex of $H_m$, then we additionally lose~$v$ itself, and thus $\operatorname{core}_2(H_m-v)$ has at least $\abs{V_{\geq 3}(H_m)} -4$ branching vertices. 
    Since each $P_i$ is added outside the closed neighborhoods of the branching vertices of $H_{i-1}$, $T_i$ contains no branching vertex of $H_{i-1}$.
Furthermore, by the definition of $T_i$, the only vertices of $T_i$ that can be branching vertices of $H_m$ are those created by the addition of $P_i$.
Thus, each $T_i$ contains at most~$2$ branching vertices of $H_m$. 
    By Lemma~\ref{prop:boundingchord}, $\abs{T_i} \le 10$ for every $i \in I$.
Therefore, for each $i\in I$,
    \[
    \abs{V_{\ge3}(\operatorname{core}_2(H_m-T_i))} \ge \abs{V_{\ge3}(H_m)} - \bigl((3+1)\cdot2+2\cdot (|T_i|-2)\bigr) \ge \abs{V_{\ge3}(H_m)} - 24.
    \]
    Thus, 
    $\abs{V_{\ge3}(H)}  = \abs{V_{\ge3}(\operatorname{core}_2(H_m - \bigcup_{i\in I}T_i))}  \ge 2(m-1) -24\abs{I}$.
\end{proof}

\begin{lemma}\label{lem:applyingskeleton}
    If $m\ge 104k\log k+132k+1$, then~$G$ contains an induced packing of~$k$ even cycles.
\end{lemma}
\begin{proof}
    Let~$I$ be the set of all non-matching indices.
     Let~$H$ be the graph obtained from $H_m$ by removing $\bigcup_{i\in I}T_i$, and then recursively deleting vertices of degree at most~$1$.
    By Lemma~\ref{lem:transformtoskeleton}
    $H$ is a semi-skeleton with at least $2(m-1)-24\abs{I}$ branching vertices.

    If $\abs{I}\ge 11k$, then by Lemma~\ref{lem:nonmatching}, $G$ contains an induced packing of~$k$ even cycles.
    So, we may assume that $\abs{I}< 11k$.
    Therefore, $H$ has at least 
    $2(m-1)-24(11k)\ge 208k\log k$ 
    many branching vertices.
Let $H'$ be a skeleton completion of~$H$.
Then $H'$ is a skeleton with at least $208k\log k$ branching vertices.
Then by Lemma~\ref{lem:manybranching}, $G$ contains an induced packing of~$k$ even cycles.
\end{proof}

\subsection{Parity-breaking extensions}

In the next step, we add parity-breaking paths on top of a maximal bipartite ear decomposition.
Before defining it, we observe that two parity-breaking paths whose endpoints are disjoint cannot share an internal vertex, as in such a case, we can extract a parity-preserving path.
We prove a slightly stronger statement.

    Recall that $\{(P_i, M_i, J_i, T_i, S_i) \}_{i\in [m]}$ is a maximal bipartite ear decomposition in~$G$, and $H_m=\bigcup_{j\in [m]}P_j$ and $\mathcal{S}_m=\bigcup_{j\in [m]}S_j$.
Note that $\mathcal{S}_m$ meets all the chords of $H_m$.
Therefore, $H_m-\mathcal{S}_m$ is an induced subgraph of~$G$.
Let 
    \[H^*=H_m-\mathcal{S}_m \qquad \text{and} \qquad G^*=G_m-N_{G-(V(H_m)\setminus \mathcal{S}_m})[\mathcal{S}_m].\]

\begin{lemma}\label{lem:paritybreakingbasic}
    Every $H^*$-path in $G^*$ is parity-breaking for $H_m$ in~$G$.
\end{lemma}
\begin{proof}
    Let~$Q$ be an $H^*$-path in $G^*$. 
    If~$Q$ is parity-preserving for $H_m$, then $\{(P_i, M_i, J_i, T_i, S_i)\}_{i\in[m]}$ could be extended by adding a parity-preserving path, contradicting its maximality.
Thus, $Q$ is parity-breaking for $H_m$.
\end{proof}

\begin{lemma}\label{lem:paritybreaking}
    Let~$Q$ be an $H^*$-path in $G^*$.
    Then there is no $(H^*\cup Q)$-path~$R$ in $G^*$ such that one endpoint of~$R$ is in $V(H^*)\setminus V(Q)$ and the other endpoint is an internal vertex of~$Q$. 
\end{lemma}
\begin{proof}
    Suppose that such a path~$R$ exists. 
    Let $Q=q_1q_2 \cdots q_m$ and $R=r_1r_2 \cdots r_d$ where $r_d\in V(Q)$.
Let $\beta$ be a bipartition indicator of $H_m$. 

    Let $Q_1=q_1Qr_d$ and $Q_2=q_mQr_d$.
As~$Q$ is parity-breaking for $H_m$ by Lemma~\ref{lem:paritybreakingbasic}, without loss of generality, we may assume that 
    \begin{itemize}
        \item $\beta(q_1)\not\equiv |E(Q_1)|\pmod 2$ and $\beta(q_m)\equiv |E(Q_2)|\pmod 2$.
    \end{itemize}
    If $\beta(r_1)\not\equiv |E(R)|\pmod 2$, then $Q_1\cup R$ is parity-preserving.
If $\beta(r_1)\equiv |E(R)|\pmod 2$, then $Q_2\cup R$ is parity-preserving.
Thus, in any case, we can add one more $H_m$-extension. 
    This contradicts the assumption that 
    $\{(P_i, M_i, J_i, T_i, S_i) \}_{i\in [m]}$ is a maximal bipartite ear decomposition.
\end{proof}

Let~$F$ be a linear forest, and let~$P$ and~$Q$ be two~$F$-paths
with endpoints~$p_1$, $p_2$ and~$q_1$, $q_2$, respectively.
For an~$F$-path~$X$ whose endpoints lie in the same component~$R$ of
$F$, let $I_X$ denote the unique subpath of~$R$ joining the endpoints of~$X$.

We say that~$P$ and~$Q$ are \emph{crossing with respect to~$F$}
if their endpoints lie in the same component~$R$ of~$F$, and, after
possibly reversing the ordering of~$R$ and exchanging~$P$ and~$Q$,
\[p_1 <_R q_1 <_R p_2 <_R q_2.\]
We say that~$P$ and~$Q$ are \emph{strongly crossing with respect to~$F$}
if they are crossing with respect to~$F$ and there is no edge in~$R$
between $\{p_1,p_2\}$ and $\{q_1,q_2\}$.

We say that~$P$ and~$Q$ are \emph{nested with respect to~$F$}
if their endpoints lie in the same component of~$F$, and either
$I_P$ is a subpath of $I_Q$ or $I_Q$ is a subpath of $I_P$.
We say that~$P$ and~$Q$ are \emph{strongly nested with respect to~$F$}
if they are nested with respect to~$F$ and there is no edge in~$F$
between $\{p_1,p_2\}$ and $\{q_1,q_2\}$.

We define a \emph{parity-breaking extension}
$\{(L_j,Y_j)\}_{j\in[q]}$
as follows.

\begin{itemize}
    \item Suppose that $L_1,\ldots,L_{i-1}$, $Y_1,\ldots,Y_{i-1}$ have been defined.
    At step~$i$, we apply an applicable rule in the following list.
    \begin{enumerate}[label=(\arabic*)]
        \item\label{prior1} If~$G$ contains an $H^*$-path contained in $G^*-\bigcup_{j\in [i-1]}Y_j$ that connects distinct components of $H^*-\bigcup_{j\in [i-1]}Y_j$, then we choose a shortest such path and call it $L_i$.
        \item\label{prior2} If~$G$ contains two vertex-disjoint $H^*$-paths contained in $G^*-\bigcup_{j\in [i-1]}Y_j$ that are strongly crossing with respect to $H^*-\bigcup_{j\in [i-1]}Y_j$, then we choose such a pair.
We repeatedly shorten the obtained paths as follows.
If one of them, say~$P$, has a chord $xy$, then we replace $xPy$ with $xy$.
Let $L_i$ be the union of the resulting paths.
        \item\label{prior3} If~$G$ contains two vertex-disjoint $H^*$-paths contained in $G^*-\bigcup_{j\in [i-1]}Y_j$ that are strongly nested with respect to $H^*-\bigcup_{j\in [i-1]}Y_j$, then we choose such a pair.
        We repeatedly shorten the obtained paths as follows.
If one of them, say~$P$, has a chord $xy$, then we replace $xPy$ with $xy$.
Let $L_i$ be the union of the resulting paths.
    \end{enumerate}
\item Let $Y_i=N_{H^*}[V_1(L_i)]$.
   If none of the above rules applies, then we stop the procedure.
\end{itemize}

We say that a parity-breaking extension $\{(L_j, Y_j) \}_{j\in [q]}$ is \emph{maximal} if there is no parity-breaking extension $\{(L_j', Y_j') \}_{j\in [q']}$ such that $q< q'$ and for all $i\in [q]$, $L_i=L_i'$. 

In the remainder of this section, we fix a maximal parity-breaking extension $\{(L_j, Y_j) \}_{j\in [q]}$.
Let 
\[\widehat{H}=H^*\cup \left(\bigcup_{i\in [q]}L_i\right).\]
We first observe that 
$\widehat{H}$ is an induced subgraph of~$G$.

\begin{lemma}\label{lem:inducedHstar}
    The subgraphs $L_1,\ldots,L_q$ are pairwise vertex-disjoint, and $\widehat{H}$ is an induced subgraph of~$G$.
\end{lemma}
\begin{proof}
Recall that $H^*$ is an induced subgraph of~$G$.
Moreover, as we choose $L_i$ with the minimum number of vertices in each iteration,
every component of $L_i$ is an induced path in~$G$.
Also, when $L_i$ is the union of two paths, say $L_i'$ and $L_i''$,
there is no edge between $L_i'$ and $L_i''$; otherwise, this would
contradict Lemma~\ref{lem:paritybreaking}.

We claim that $L_1,\ldots,L_q$ are pairwise vertex-disjoint.
By construction, for distinct $i,j\in[q]$, we have $V_1(L_i)\cap V_1(L_j)=\emptyset$.
Thus, if $L_i$ and $L_j$ intersect, then they have a common vertex in $V_2(L_i)\cap V_2(L_j)$.

Suppose for a contradiction that for some $i\ne j$, $V_2(L_i)\cap V_2(L_j)$ contains a vertex~$z$. 
Let $x_j$ be an endpoint of the component of $L_j$ containing~$z$, and let~$v$ be the first vertex of $L_j$, starting from $x_j$, that belongs to $L_i$.
Let $L_i'$ be the component of $L_i$ containing~$v$.
Then $x_jL_jv$
is an $(H^*\cup L_i')$-path in $G^*$ with one endpoint
$x_j\in V(H^*)\setminus V(L_i')$ and the other endpoint
$v\in\internal(L_i')$, contradicting
Lemma~\ref{lem:paritybreaking}.
Thus, $L_1,\ldots,L_q$ are vertex-disjoint.

We next prove that for each $i\in [q]$, there is no edge between $V_2(L_i)$ and $\widehat{H}-V(L_i)$.
Suppose for contradiction that there is an edge $ab$ with $a\in V_2(L_i)$ and $b\in V(\widehat{H})\setminus V(L_i)$.
Let $L_i'$ be the component of $L_i$ containing~$a$.
If~$b$ is a vertex of $H^*$, then $ab$ is an $(H^*\cup L_i')$-path in $G^*$.
If~$b$ is a vertex of $V(\internal(L_j))$ for some $j\neq i$, then the union of $ab$ and a subpath of $L_j$ from~$b$ to one of its endpoint is an $(H^*\cup L_i')$-path.
This contradicts Lemma~\ref{lem:paritybreaking}.

We deduce that $\widehat{H}$ is an induced subgraph of~$G$.
\end{proof}

Now, we want to show that if~$q$ is large, then we can find an induced packing of~$k$ even cycles.
Recall the definition of $\aux(F,\mathcal Q)$ defined in Section~\ref{sec:pathmatching}.
To apply Proposition~\ref{prop:many-links}, we show that the corresponding $\aux(F,\mathcal Q)$ satisfies the consecutive serial property. 

\begin{lemma}\label{lem:consecutiveserial}
    Let $\mathcal{L}$ be the family of all components of graphs in $\{L_j:j\in [q]\}$.
The path-matching graph $\aux(H^*,\mathcal{L})$ satisfies the consecutive serial property.
\end{lemma}
\begin{proof}
For each $i\in[q]$, let $\mathcal{L}_i$ be the family of all components
of the graphs in $\{L_j:j\in[i]\}$, and let
$J_i:=\aux(H^*,\mathcal{L}_i)$.
We prove by induction on~$i$ that $J_i$ satisfies the consecutive
serial property.

The claim is trivial for $i=1$.
Suppose that $i\ge2$ and that
$J_{i-1}$ satisfies the consecutive serial property.
Set
\[
F_{i-1}:=H^*-\bigcup_{j\in[i-1]}Y_j.
\]
Since every endpoint of a path in $\mathcal{L}_{i-1}$ belongs to
$\bigcup_{j\in[i-1]}Y_j$, no component of $F_{i-1}$ contains an
endpoint of a link of $J_{i-1}$.

We distinguish the rule used to construct $L_i$.

Suppose first that $L_i$ is chosen by Rule~\ref{prior1}, and let~$x$ and~$y$ be its endpoints.
If~$x$ and~$y$ lie in different components
of $H^*$, then the link $xy$ is not a same-path link and hence cannot
belong to a consecutive serial pair.

Thus, suppose that~$x$ and~$y$ lie in the same component~$R$ of $H^*$.
Since they lie in distinct components of $F_{i-1}$, the subpath
$xRy$ contains an endpoint of a link of $J_{i-1}$ in its interior.
Therefore, the new link $xy$ cannot form a consecutive serial pair
with a link of $J_{i-1}$.
As adding a new link cannot create a new
consecutive serial pair consisting of two old links, $J_i$ satisfies
the consecutive serial property.

Now suppose that $L_i$ is chosen by Rule~\ref{prior2} or
Rule~\ref{prior3}.
Let~$P$ and~$Q$ be the two components of $L_i$,
and let~$e$ and~$f$ be their corresponding links.
By the definition
of these rules, the four endpoints of~$e$ and~$f$ lie in the same
component~$R$ of $F_{i-1}$, and~$e$ and~$f$ are crossing or nested.
In particular, $e$ and~$f$ are not serial, and hence they do not form
a consecutive serial pair with each other.

It remains to consider a pair consisting of a new link, say~$e$, and
an old link~$g$ of $J_{i-1}$.
Suppose that~$e$ and~$g$ are serial.
Since~$R$ contains no endpoint of an old link, both endpoints of~$g$
lie on the same side of~$R$.
On the other hand, since~$e$ and~$f$
are crossing or nested, some endpoint of~$f$ lies strictly between
the extreme endpoints of~$e$ and~$g$.
Hence~$e$ and~$g$ are not
consecutive serial.
The same argument applies with~$e$ and~$f$
interchanged.

Therefore, no new consecutive serial pair is created when passing
from $J_{i-1}$ to $J_i$.
Thus $J_i$ satisfies the consecutive serial
property.

By induction, $J_q=\aux(H^*,\mathcal L)$ satisfies the consecutive
serial property.
\end{proof}

\begin{lemma}\label{lem:caselargeq}    
Let $m'$ be the number of components of $H^*$.
If $q>11m'+23g_{\theta}(k)$, then~$G$ contains an induced packing of~$k$ even cycles.
\end{lemma}
\begin{proof}
    Suppose that $q>11m'+23g_{\theta}(k)$.
    Let $\mathcal{L}$ be the family of all components of graphs in $\{L_j:j\in [q]\}$, and
    $N=N(\mathcal{L})$.
    By applying Proposition~\ref{prop:many-links} to $\aux(H^*, \mathcal{L})$, 
    we have that $\aux(H^*, \mathcal{L})$ contains~$k$ pairwise vertex-disjoint even-$N$ cycles, say $C_1, \ldots, C_k$. 
    For each $j\in[k]$, let $D_j$ be the cycle in~$G$ obtained
from $C_j$ by replacing each link $x_iy_i\in E(C_j)\cap N$
with the corresponding component of $L_i$.

    We claim that $\{D_1, \ldots, D_k\}$ is an induced packing of even cycles in~$G$.
Note that each $C_i$ is an even-$N$ cycle, and so each $D_i$ contains an even number of paths in $\mathcal{L}$.
Since each component of $L_j$ is parity-breaking with respect to the bipartition of $H_m$ by Lemma~\ref{lem:paritybreakingbasic}, and $H^*$ inherits this bipartition,
    replacing an even number of links of $C_i$ by the corresponding paths results in an even cycle.
Thus, each $D_i$ is an even cycle.

    As $C_1,\ldots,C_k$ are pairwise vertex-disjoint and the paths
in $\mathcal{L}$ are pairwise vertex-disjoint by Lemma~\ref{lem:inducedHstar}, the cycles
$D_1,\ldots,D_k$ are pairwise vertex-disjoint.

Suppose for a contradiction that there is an edge $ab\in E(G)$
joining two distinct cycles $D,D'\in\{D_1,\ldots,D_k\}$, where
$a\in V(D)$ and $b\in V(D')$.
By Lemma~\ref{lem:inducedHstar},
$\widehat{H}$ is an induced subgraph of~$G$.
Thus, $ab$ is an edge of this
subgraph.

We first observe that $ab\in E(H^*)$.
If $ab$ is an edge of some $L\in\mathcal{L}$, then one of $D,D'$ containing an internal vertex of~$L$ must contain the whole path
$L$, since~$L$ replaces a single link of the corresponding auxiliary cycle.
This contradicts the fact that~$D$ and $D'$ are vertex-disjoint.

Let~$C$ and $C'$ be the cycles in $\aux(H^*, \mathcal{L})$ corresponding to
$D$ and $D'$, respectively.
Since $ab\in E(H^*)$ but $ab\notin E(C)$, the cycle~$C$ must use a link incident with
$a$.
Thus, $a$ is an endpoint of a path in $\mathcal{L}$.
Similarly, $b$ is an endpoint of a path in $\mathcal{L}$.

By the construction of $\mathcal{L}$, endpoints of distinct paths
in $\mathcal{L}$ are not adjacent in $H^*$.
Since $ab\in E(H^*)$,
the vertices~$a$ and~$b$ must therefore be the endpoints of the
same path $L\in\mathcal{L}$.
Consequently, the link corresponding to~$L$ belongs to both~$C$ and $C'$, contradicting the fact that~$C$ and $C'$ are
vertex-disjoint.

We conclude that  $\{D_1, \ldots, D_k\}$ is an induced packing of even cycles.
\end{proof}

By Lemma~\ref{lem:caselargeq}, we may assume that
$q\le 11m'+23g_{\theta}(k)$, where $m'$ is the number of components of $H^*$.

Note that there may still be crossing or nested pairs of $H^*$-paths.
However, any such pair is neither strongly crossing nor strongly nested.
Thus, the endpoints of the two paths are at distance at most~$1$ in $H^*$.
Consequently, after removing the closed $H^*$-neighborhoods of their
endpoints, no remaining $H^*$-path can cross or be nested with either
of the two paths.

Therefore, we can greedily select such pairs.
This either yields many
pairwise separated pairs, or, after removing the closed neighborhoods
of the endpoints of the selected pairs, eliminates all remaining
crossing and nested pairs.

\begin{lemma}\label{lem:removingcrossingnested}
    Either~$G$ contains an induced packing of~$k$ even cycles, or it contains a vertex set $Z\subseteq V(H^*)$ of size at most $4k$ such that there are no two internally disjoint $H^*$-paths contained in $G^*-\bigcup_{i\in [q]}Y_i-N_{H^*}[Z]$ that are crossing or nested with respect to  $H^*-\bigcup_{i\in [q]}Y_i-N_{H^*}[Z]$.
\end{lemma}
\begin{proof}
Let
\[
F_0=H^*-\bigcup_{i\in[q]}Y_i
\qquad\text{and}\qquad
G_0=G^*-\bigcup_{i\in[q]}Y_i.
\]
By the termination of the priority procedure, $G_0$ contains no pair
of vertex-disjoint $H^*$-paths that are strongly crossing or strongly
nested with respect to $F_0$.

We construct a sequence of pairs of $H^*$-paths greedily.
Suppose that pairs
$(P_1,Q_1),\ldots,(P_{s-1},Q_{s-1})$
have been chosen.
For each $j\in[s-1]$, let
$A_j=V_1(P_j)\cup V_1(Q_j)$, and
$Z_{s-1}=\bigcup_{j\in[s-1]}A_j$.
Set
\[
F_{s-1}=F_0-N_{H^*}[Z_{s-1}]
\quad\text{and}\quad
G_{s-1}=G_0-N_{H^*}[Z_{s-1}].
\]
If $G_{s-1}$ contains two internally disjoint $H^*$-paths $P_s$ and $Q_s$
that are crossing or nested with respect to $F_{s-1}$, then we choose
such a pair and continue.
Otherwise, we stop the procedure.

For each chosen pair $(P_s,Q_s)$, let $I(P_s)$ and $I(Q_s)$ denote
the subpaths of $F_{s-1}$ between the endpoints of $P_s$ and $Q_s$,
respectively.
Since every $H^*$-path in $G^*$ is parity-breaking for $H_m$, both
$P_s\cup I(P_s)$ and $Q_s\cup I(Q_s)$
are odd cycles.
As $P_s$ and $Q_s$ are crossing or nested, the symmetric difference
of these two cycles is a cycle; denote it by $D_s$.
Since it is the symmetric difference of two odd cycles, $D_s$ is even.

We claim that the cycles obtained in distinct iterations are
vertex-disjoint and anticomplete.
We first observe the following.

\begin{claim}\label{claim:serial}
For $r<s$, every $H^*$-path contained in
$G_r$ is serial with both $P_r$ and $Q_r$ whenever their endpoints
lie in the same component of $F_0$.
\end{claim}
\begin{clproof}
Let~$R$ be such an $H^*$-path.
Suppose, for instance, that
$R$ is crossing or nested with $P_r$.
Since the endpoints of~$R$
avoid $N_{H^*}[A_r]$, no endpoint of~$R$ is equal or adjacent to an
endpoint of $P_r$.
Moreover, Lemma~\ref{lem:paritybreaking} implies
that~$R$ and $P_r$ are vertex-disjoint.
Hence~$R$ and $P_r$ are
strongly crossing or strongly nested with respect to $F_0$, a
contradiction to the termination of the priority procedure.
The same argument applies to $Q_r$.
\end{clproof}

Let $r<s$.
By Claim~\ref{claim:serial}, each of $P_s$ and $Q_s$ is serial with both $P_r$
and $Q_r$.
Since $P_s$ and $Q_s$ are crossing or nested with each
other, $I(P_s)\cup I(Q_s)$
and 
$I(P_r)\cup I(Q_r)$ are vertex-disjoint.
Furthermore, since the endpoints of $P_s$ and $Q_s$ avoid
$N_{H^*}[A_r]$, there is no edge of $H^*$ between these two
subpaths.

It remains to consider edges involving internal vertices of the
chosen $H^*$-paths.
By Lemma~\ref{lem:paritybreaking}, an internal
vertex of one of $P_r,Q_r$ has no neighbor in $H^*$ outside its own
endpoints.
Moreover, there is no edge from an internal vertex of one
of $P_r$ and $Q_r$ to an internal vertex of one of $P_s$ and $Q_s$; otherwise,
by following the latter path from one of its endpoints up to such an
edge, we obtain an $(H^*\cup P_r)$-path or an $(H^*\cup Q_r)$-path
from $H^*$ to an internal vertex of $P_r$ or $Q_r$, contradicting
Lemma~\ref{lem:paritybreaking}.

Thus, $D_r$ and $D_s$ are vertex-disjoint and there is no edge between them.
Therefore, if the procedure chooses~$k$ pairs, then 
$D_1,\ldots,D_k$
form an induced packing of~$k$ even cycles in~$G$.
We may therefore assume that the procedure stops after $t<k$ iterations. 

Let $Z=\bigcup_{j\in[t]}A_j$.
Since $\abs{A_j}\le 4$ for every $j\in[t]$, we have $\abs{Z}\le 4t\le 4(k-1)<4k$.
By the stopping condition, there are no two $H^*$-paths contained in
\[
G_0-N_{H^*}[Z]
=
G^*-\bigcup_{i\in[q]}Y_i-N_{H^*}[Z]
\]
that are crossing or nested with respect to
\[
F_0-N_{H^*}[Z]
=
H^*-\bigcup_{i\in[q]}Y_i-N_{H^*}[Z].
\]
This proves the lemma.
\end{proof}

\subsection{Bridges}\label{subsec:bridge}

Let~$Z$ be the set obtained by applying Lemma~\ref{lem:removingcrossingnested}.
Let 
\[
H^{\dagger}=H^*-\bigcup_{i\in [q]}Y_i-N_{H^*}[Z]
\quad\text{and}\quad
G^{\dagger}=G^*-\bigcup_{i\in [q]}Y_i-N_{H^*}[Z].
\]

We now consider the $H^{\dagger}$-bridges of $G^{\dagger}$.
We first show that every $H^{\dagger}$-bridge in $G^{\dagger}$ has at most two
attachments in $H^{\dagger}$.
We then show that every even cycle of $G^{\dagger}$ is contained in a single $H^{\dagger}$-bridge.
Finally, we greedily consider the $H^{\dagger}$-bridges containing
an even cycle.

\begin{lemma}\label{lem:path-same-component}
Every $H^\dagger$-path in $G^\dagger$ has both endpoints in the same
component of $H^\dagger$.
\end{lemma}
\begin{proof}
Suppose for a contradiction that there is an $H^\dagger$-path~$Q$
in $G^\dagger$ whose endpoints~$x$ and~$y$ lie in distinct components
of $H^\dagger$.
Since $G^\dagger$ is a subgraph of 
$G^*-\bigcup_{i\in[q]}Y_i$,
$Q$ is also an $H^*$-path contained in $G^*-\bigcup_{i\in[q]}Y_i$.

Let
$F=H^*-\bigcup_{i\in[q]}Y_i$.
If~$x$ and~$y$ lie in distinct components of~$F$, then~$Q$ is a path
described in Rule~\ref{prior1}, contradicting the termination of the
parity-breaking extension.

Thus, $x$ and~$y$ lie in the same component~$R$ of~$F$.
Since they lie in distinct components of
$H^\dagger=F-N_{H^*}[Z]$,
the subpath $xRy$ meets $N_{H^*}[Z]$.
As $x,y\notin N_{H^*}[Z]$ and~$R$ is a path, $xRy$ contains a vertex
$z\in Z$ in its interior.

By the construction of~$Z$, the vertex~$z$ is an endpoint of one of
the two $H^*$-paths in some pair selected in
Lemma~\ref{lem:removingcrossingnested}; let~$P$ be this path, and let
$z'$ be its other endpoint.
In particular, $z,z'\in Z$.

If $z'\in V(xRy)$, then
$I_P$ is a subpath of $xRy=I_Q$,
and thus~$P$ and~$Q$ are nested with respect to~$F$.
If $z'\notin V(xRy)$, then~$P$ and~$Q$ are crossing with respect to
$F$.

Moreover, since $x,y\notin N_{H^*}[Z]$, no endpoint of~$Q$ is adjacent
in $H^*$ to an endpoint of~$P$.
Thus, $P$ and~$Q$ are strongly nested
or strongly crossing with respect to~$F$.
By Lemma~\ref{lem:paritybreaking}, $P$ and~$Q$ are vertex-disjoint.
Hence this contradicts the termination of
Rules~\ref{prior2} and~\ref{prior3}.

Therefore, every $H^\dagger$-path in $G^\dagger$ has both endpoints
in the same component of $H^\dagger$.
\end{proof}

\begin{lemma}\label{lem:one-intersection}
Every even cycle of $G^\dagger$ intersects $H^\dagger$ in at
most one vertex. 
\end{lemma}
\begin{proof}
Suppose for a contradiction that~$U$ is an even cycle of $G^\dagger$
such that
$\abs{V(U)\cap V(H^\dagger)}\ge 2$.

Let $\mathcal P$ be the family of maximal subpaths of~$U$ whose
endpoints are in $H^\dagger$ and whose internal vertices are outside
$H^\dagger$.
By Lemma~\ref{lem:path-same-component}, every $H^\dagger$-path in $G^\dagger$ has both endpoints in the
same component of $H^\dagger$.
Thus all vertices of
$V(U)\cap V(H^\dagger)$ lie in the same component, say~$R$, of $H^\dagger$.

For each $P\in\mathcal P$, let $I_P$ denote the unique subpath of~$R$
joining the endpoints of~$P$.

Suppose first that $\abs{\mathcal P}=1$, say $\mathcal P=\{P\}$.
Then
$U=P\cup I_P$.
Since every $H^*$-path in $G^*$ is parity-breaking for $H_m$, the cycle $P\cup I_P$ is odd, a contradiction.

Thus, $\abs{\mathcal P}\ge2$.
Since~$U$ is a cycle, there exist two paths~$P$ and $Q\in\mathcal P$ such
that $I_P$ and $I_Q$ are either crossing or one is contained in the
other.
Moreover, $P$ and~$Q$ are internally disjoint, as they are distinct
maximal subpaths of the cycle~$U$.

Hence, $P$ and~$Q$ are two internally disjoint $H^*$-paths in
$G^\dagger$ that are crossing or nested with respect to
$H^\dagger$, contradicting Lemma~\ref{lem:removingcrossingnested}.
Therefore,
we have $\abs{V(U)\cap V(H^\dagger)}\le 1$.
\end{proof}

\begin{lemma}\label{lem:singlebridge}
    Every even cycle of $G^\dagger$ is contained in a single
$H^\dagger$-bridge~$F$, where $N_{G^{\dagger}}[V(F)\cap V(H^{\dagger})]$ meets every even cycle contained in~$F$.
\end{lemma}
\begin{proof}
    After deleting the possible unique vertex of the cycle in
$H^\dagger$, the remaining part of the cycle is connected and hence
contained in a single component of
$G^\dagger-V(H^\dagger)$.

Moreover, since $G-N_G[V(C)]$ contains no even cycle, every even cycle
$D$ contained in $G^\dagger-V(H^\dagger)$ has a neighbor in~$C$.
Since no vertex of $G^\dagger$ is adjacent to a vertex of
$C-V(H^\dagger)$, such a neighbor must belong to $H^\dagger$.
Therefore, the component of $G^\dagger-V(H^\dagger)$ containing~$D$,
together with its edges to $H^\dagger$, forms an $H^\dagger$-bridge
containing~$D$.

Let
$A(F)=V(F)\cap V(H^\dagger)$.
We claim that
$N_{G^\dagger}[A(F)]$
meets every even cycle contained in~$F$.

Suppose for a contradiction that~$F$ contains an even cycle~$D$
disjoint from $N_{G^\dagger}[A(F)]$. 
As observed above, every even cycle contained in
$G^\dagger-V(H^\dagger)$ has a neighbor in $H^\dagger$.
Thus, some vertex of~$D$ has a neighbor in $H^\dagger$.
Since~$D$ is contained in the $H^\dagger$-bridge~$F$, every such
neighbor in $H^\dagger$ belongs to $A(F)$.
Therefore,
$V(D)\cap N_{G^\dagger}[A(F)]\neq\emptyset$,
a contradiction.
\end{proof}

\begin{lemma}\label{lem:bridge-attachments}
Every $H^\dagger$-bridge of $G^\dagger$ has at most two attachments in $H^\dagger$.
\end{lemma}
\begin{proof}
Suppose for a contradiction that there is an $H^\dagger$-bridge~$B$
of $G^\dagger$ having at least three distinct attachments
$a,b,c\in V(H^\dagger)$.

 Let~$K$ be the component of
$G^\dagger-V(H^\dagger)$ corresponding to~$B$.
Choose neighbors $a',b',c'\in V(K)$ of $a,b,c$, respectively.
Since~$K$ is connected, there is a tree~$T$ contained in
$K+\{aa',bb',cc'\}$
that connects $a,b,c$ and is minimal with this property.
In particular, $a,b,c$ are leaves of~$T$.

Let~$P$ be the unique path from~$a$ to~$b$ in~$T$.
Since
$V(G^\dagger)\setminus V(H^\dagger)
   \subseteq V(G^*)\setminus V(H^*)$,
$P$ is an $H^*$-path contained in $G^*$.

Let~$R$ be the path in~$T$ from~$c$ to the first vertex~$w$ of~$P$
met when traversing~$T$ from~$c$ towards~$P$.
Since $a,b,c$ are leaves of~$T$, the vertex~$w$ is an internal vertex
of~$P$.
Moreover,
$c\in V(H^*)\setminus V(P)$,
and all internal vertices of~$R$ lie outside $H^*\cup P$.
Hence~$R$ is an $(H^*\cup P)$-path in $G^*$ with one endpoint in
$V(H^*)\setminus V(P)$ and the other endpoint in $\internal(P)$.
This contradicts Lemma~\ref{lem:paritybreaking}.

Therefore, every $H^\dagger$-bridge of $G^\dagger$ has at most two
attachments in $H^\dagger$.
\end{proof}

\begin{lemma}\label{lem:final-greedy}
Either~$G$ contains an induced packing of~$k$ even cycles, or there
exists a set $W\subseteq V(H^\dagger)$ of size at most $2(k-1)$ such
that
$G^\dagger-N_{G^\dagger}[W]$
contains no even cycle.
\end{lemma}
\begin{proof}
We greedily construct even cycles $D_1,D_2,\ldots$ and
$H^\dagger$-bridges $F_1,F_2,\ldots$ as follows.

Suppose that $D_1,\ldots,D_{i-1}$ and
$F_1,\ldots,F_{i-1}$ have been chosen.
For each $j\in[i-1]$, let
$A_j=V(F_j)\cap V(H^\dagger)$,
and $W_{i-1}=\bigcup_{j\in[i-1]}A_j$.
If
$G^\dagger-N_{G^\dagger}[W_{i-1}]$
contains an even cycle, choose one and call it $D_i$.
By Lemma~\ref{lem:one-intersection} and the discussion following it,
$D_i$ is contained in a single $H^\dagger$-bridge; let $F_i$ be such
a bridge.
If no such even cycle exists, we stop the procedure.

We first claim that the cycles chosen by this procedure form an
induced packing.
Let $1\le r<s$.
Since
$N_{G^\dagger}[A_r]\subseteq
N_{G^\dagger}[W_{s-1}]$,
we have
$V(D_s)\cap N_{G^\dagger}[A_r]=\emptyset$.
Moreover, $F_r$ and $F_s$ are distinct.
Indeed,
$N_{G^\dagger}[A_r]$ meets every even cycle contained in $F_r$,
whereas $D_s$ is disjoint from this set.

The interiors of two distinct $H^\dagger$-bridges lie in distinct
components of $G^\dagger-V(H^\dagger)$, and hence there is no edge
between them.
Also,
$V(D_r)\cap V(H^\dagger)\subseteq A_r$.
Thus, no vertex of $D_s$ is equal or adjacent to a vertex of
$D_r\cap V(H^\dagger)$.

Conversely, suppose that a vertex
$x\in V(D_r)\setminus V(H^\dagger)$ is adjacent to a vertex
$y\in V(D_s)\cap V(H^\dagger)$.
Then~$y$ is an attachment of $F_r$,
so $y\in A_r$.
This is a contradiction.
Therefore, $D_r$ and $D_s$ are
vertex-disjoint and there is no edge between them.

It follows that, if the procedure chooses~$k$ cycles, then
$D_1,\ldots,D_k$ form an induced packing of~$k$ even cycles.
We may therefore assume that the procedure stops after $t<k$
iterations. 

Let
$W=\bigcup_{i\in[t]}A_i$.
By Lemma~\ref{lem:bridge-attachments}, each $H^\dagger$-bridge has at
most two attachments.
Therefore,
\[
|W|
   \le \sum_{i\in[t]}|A_i|
   \le 2t
   \le 2(k-1).
\]
By the stopping condition,
$G^\dagger-N_{G^\dagger}[W]$
contains no even cycle.
This proves the lemma.
\end{proof}

\subsection{Proof of Theorem~\ref{thm:main2}}\label{subsec:mainproof}

Now, we complete the proof of Theorem~\ref{thm:main2}.

\begin{proof}[Proof of Theorem~\ref{thm:main2}]
Suppose first that $\abs{V(C)}=4$.
Then, since
$G-N_G[C]$
contains no even cycle, we may simply take $X=V(C)$.
Thus, we may
assume that $\abs{V(C)}\ge 6$.
Since~$C$ is a shortest even cycle, $G$
is $C_4$-subgraph-free.

Let $\{(P_i,M_i,J_i,T_i,S_i):i\in[m]\}$ be a  maximal bipartite ear decomposition in~$G$.
By Lemma~\ref{lem:applyingskeleton},
  if $m\ge 104k\log k+132k+1$, then~$G$ contains an induced packing of~$k$ even cycles. 
  Thus, we may assume that $m<104k\log k+132k+1$.

 Let $\mathcal S_m$,
 $H^*$, and $G^*$ be defined as before.
Let $m'$ be the number of components of $H^*$.
Since $m=O(k\log k)$, we have $m'=O(k\log k)$.
 
We now consider the parity-breaking extension $\{(L_j,Y_j)\}_{j\in[q]}$, and let $\mathcal L$ be the family of all
components of the graphs in $\{L_i:i\in[q]\}$.
Let
$r=|\mathcal L|$.
By Lemma~\ref{lem:consecutiveserial}, the path-matching graph $\aux(H^*,\mathcal L)$
satisfies the consecutive serial property.

Recall that $g_\theta(k)=O(k\log k)$ denotes the Erd\H{o}s--P\'osa
gap for $\theta$-minor models.
By Lemma~\ref{lem:caselargeq},
if
$r>11m'+23g_\theta(k)$,
then~$G$ contains an induced packing of~$k$ even cycles.
Thus, we may assume that
$r\le 11m'+23g_\theta(k)=O(k\log k)$.
Let $A=\bigcup_{L\in\mathcal L}V_1(L)$.
Since every member of $\mathcal L$ is a path,
$\abs{A}\le 2r=O(k\log k)$.
By the definition of the sets $Y_i$,
$\bigcup_{i\in[q]}Y_i=N_{H^*}[A]$.

We next apply Lemma~\ref{lem:removingcrossingnested}.
If it yields an
induced packing of~$k$ even cycles, we are done.
Otherwise, there is
a set
$Z\subseteq V(H^*)$ of size at most $4k$
such that, defining
\[
H^\dagger
 =
 H^*-\bigcup_{i\in[q]}Y_i-N_{H^*}[Z]
\quad\text{and}\quad
G^\dagger
 =
 G^*-\bigcup_{i\in[q]}Y_i-N_{H^*}[Z],
\]
there are no two internally disjoint $H^*$-paths in $G^\dagger$
that are crossing or nested with respect to $H^\dagger$.
Finally, by Lemma~\ref{lem:final-greedy}, either~$G$ contains an induced packing of~$k$ even cycles, or there
exists a set $W\subseteq V(H^\dagger)$ of size at most $2(k-1)$ such
that
$G^\dagger-N_{G^\dagger}[W]$
contains no even cycle.

Let
$X=\mathcal S_m\cup A\cup Z\cup W$.
Then we have 
\[
\begin{aligned}
|X|
&\le |\mathcal S_m|+|A|+|Z|+|W|\\
&=O(k\log k)+O(k\log k)+4k+2(k-1)=O(k\log k).
\end{aligned}
\]
By the construction, $G-N_G[X]$ has no even cycles, as desired.
\end{proof}

We explain how to find one of the two outcomes in polynomial time.
We use the following algorithm.

\begin{theorem}[Derigs~\cite{Derigs1985}]\label{thm:alg-paritypath}
    Given a graph~$G$ and $a,b\in V(G)$, one can find a shortest odd $(a,b)$-path in time $O(|V(G)|^2)$ if one exists.
Similarly, one can find a shortest  even $(a,b)$-path in time $O(|V(G)|^2)$ if one exists.
\end{theorem}

\medskip \noindent
\textbf{Bipartite ear decompositions.}
We claim that a maximal bipartite ear decomposition \[\{(P_i, M_i, J_i, T_i, S_i) \}_{i\in [m]}\] can be constructed in polynomial time.
To find a shortest path among all $H_i$-extensions contained in $G-N_{G-(V(H_i)\setminus \mathcal{S}_i)}[\mathcal{S}_i]$, we guess two vertices~$a$ and~$b$ in $H_i-\mathcal{S}_i$.
If~$a$ and~$b$ are in the same bipartition of $H_i$, then we find a shortest even $(a,b)$-path in 
\[(G-E(H_i))-N_{G-(V(H_i)\setminus \mathcal{S}_i)}[\mathcal{S}_i]-V(H_i)\setminus \{a,b\}\]
using Theorem~\ref{thm:alg-paritypath}.
If they are in the different bipartition, then we find a shortest odd $(a,b)$-path.
Among all such paths, we choose one of minimum length.
Thus, each $P_{i+1}$ can be found in time $O(|V(G)|^2\cdot |V(G)|^2)=O(|V(G)|^4)$.
The vertex cover $T_{i+1}$ of $J_{i+1}$ can be constructed by taking all vertices of degree at least~$2$ in $J_{i+1}$, together with all vertices of degree~$1$ whose unique neighbor also has degree~$1$.
This can be done in time $O(|V(G)|^2)$.
As $m\le |V(G)|$, the maximal bipartite ear decomposition can be constructed in time $O(|V(G)|^5)$.

During the construction, we keep track of the non-matching indices.
Let~$I$ denote the set of non-matching indices.
Assume $\abs{I}\ge 11k$.
By taking a subset if necessary, we may assume that $\abs{I}=11k$.
Then following the proof of Lemma~\ref{lem:nonmatching}, we choose an even cycle $D_i$ for each $i\in I$, and construct an auxiliary graph $\Gamma$.
This can be done in time $O(k|V(G)|^2)$.
Since $\Gamma$ is $10$-degenerate, we can greedily find an independent set $I'$ of $\Gamma$ of size~$k$.
Then $\{D_i:i\in I'\}$ is an induced packing of~$k$ even cycles.
This can be done in time $O(k^2)$.
Thus, we can find an induced packing of~$k$ even cycles in time $O(k|V(G)|^2+k^2)$.

So, we may assume that $\abs{I}<11k$.
By Lemma~\ref{lem:applyingskeleton},  if $m\ge 104k\log k+132k+1$, then 
     then the graph obtained from $H_m-\bigcup_{i\in I}T_i$ by
     iteratively deleting vertices of degree at most~$1$ is a semi-skeleton with at least $208k\log k$ branching vertices.
By applying Lemma~\ref{lem:manybranching} to the skeleton completion of it, $G$ contains an induced packing of~$k$ even cycles.
To find such a packing in polynomial time, it remains to give an algorithmic version of Lemma~\ref{lem:manybranching}.

    Observe that, in the proof of Lemma~\ref{lem:manybranching}, the quantity
    $\abs{F_0(H)\cup F_1(H)}$
    is at most $\abs{V(G)}$, since the chords of a skeleton form a matching.
If every corridor of~$H$ meets at most~$3$ chords in $F_0(H)$, then we delete a shortest cycle~$C$, obtain a semi-skeleton, take a skeleton completion, and apply induction.
This can be done in time $O(|V(G)|^2)$.
In the case when there is a corridor meeting at least~$4$ chords in $F_0(H)$, we take the graph $\phi(H,R)$ and apply the induction.
This can be done in time $O(|V(G)|^2)$.
    Thus, applying Lemma~\ref{lem:manybranching} takes time $O((k+|V(G)|)\cdot |V(G)|^2)$.
     
\medskip \noindent
\textbf{Parity-breaking extensions.}
Recall that  
    $H^*=H_m-\mathcal{S}_m$ and $G^*=G_m-N_{G-(V(H_m)\setminus \mathcal{S}_m})[\mathcal{S}_m]$.
    We construct a parity-breaking extension from $H^*$ and $G^*$.
By Lemma~\ref{lem:paritybreaking}, every $H^*$-path in $G^*$ is parity-breaking.
Thus, it suffices to consider $H^*$-paths without regard to parity.

    For Rule~\ref{prior1}, we guess two vertices~$a$, $b$ in distinct components of $H^*-\bigcup_{j\in [i-1]}Y_j$ and find an $(a,b)$-path in \[G^*-\bigcup_{j\in [i-1]}Y_j-(V(H^*)\setminus \{a,b\}).\]  Among all such paths, we choose one of minimum length.
    For Rules~\ref{prior2} or \ref{prior3}, we guess four vertices $a,b,c,d$ that are contained in the same component~$R$ of $H^*-\bigcup_{j\in [i-1]}Y_j$ where $a<_R b<_R c<_R d$, $ab\notin E(H)$ and $cd\notin E(H)$.
Using a max-flow algorithm, we find two vertex-disjoint paths from $\{a,b\}$ to $\{c,d\}$ in \[G^*-\bigcup_{j\in [i-1]}Y_j-(V(H^*)\setminus \{a,b,c,d\})\]
    if such paths exist, and then shorten as described.
If the two paths link~$a$ to~$c$ and~$b$ to~$d$, respectively, then they form a strongly crossing pair.
If they link~$a$ to~$d$ and~$b$ to~$c$, respectively, then they form a strongly nested pair.

    Therefore, a maximal parity-breaking extension $\{(L_j,Y_j)\}_{j\in[q]}$ can be constructed in time $O(q\cdot |V(G)|^2)=O(|V(G)|^3)$.

    Let $m'$ be the number of components of $H^*$.
    By Lemma~\ref{lem:caselargeq}, if $q>11m'+23g_{\theta}(k)$, then~$G$ contains an induced packing of~$k$ even cycles.
Following the proof of Lemma~\ref{lem:caselargeq}, it suffices to find~$k$ vertex-disjoint even-$N$ cycles in the corresponding path-matching graph.
    In the proof of Proposition~\ref{prop:many-links}, there always exist~$k$ vertex-disjoint even-$N$ cycles.
Moreover, the only nontrivial algorithmic step in that proof is to find either~$k$ vertex-disjoint $\theta$-minor models or a $\theta$-minor hitting set, which can be done in polynomial time by Theorem~\ref{thm:thetaminor}.
Therefore, the~$k$ vertex-disjoint even-$N$ cycles, and the corresponding induced packing of~$k$ even cycles in~$G$, can be found in polynomial time.

    Thus, we may assume that
    $q\le 11m'+23g_{\theta}(k)$.
Lemma~\ref{lem:removingcrossingnested} greedily packs pairs of crossing or nested paths
$(P_1,Q_1),\ldots,(P_{q'},Q_{q'})$.
As in the previous argument, this can be done in time $O(|V(G)|^3)$.
Each pair $(P_s,Q_s)$, together with the path $I(P_s)\cup I(Q_s)$, contains a $\theta_3$-minor.
Following the proof of Lemma~\ref{lem:removingcrossingnested}, we can then construct the corresponding induced packing of even cycles in polynomial time.
     
\medskip \noindent
\textbf{Bridges.}
Recall that~$Z$ is the set obtained by applying Lemma~\ref{lem:removingcrossingnested},
$H^{\dagger}=H^*-\bigcup_{i\in [q]}Y_i-N_{H^*}[Z]$
and 
$G^{\dagger}=G^*-\bigcup_{i\in [q]}Y_i-N_{H^*}[Z]$.
By Lemma~\ref{lem:singlebridge},
every even cycle of $G^\dagger$ is contained in a single
$H^\dagger$-bridge~$F$, where $N_{G^{\dagger}}[V(F)\cap V(H^{\dagger})]$ meets every even cycle contained in~$F$.
For each component~$C$ of $G^\dagger-V(H^\dagger)$, we test whether
the corresponding $H^\dagger$-bridge
$G^\dagger[V(C)\cup N_{G^\dagger}(C)]$
contains an even cycle.
By processing all such components and following the greedy procedure
in the proof of Lemma~\ref{lem:final-greedy}, we can obtain one of
the two outcomes of Lemma~\ref{lem:final-greedy} in polynomial time.

Therefore, one of the two outcomes in Theorem~\ref{thm:main2} can be found in polynomial time. 

\subsection{Corollary on tree-independence number}

We derive a consequence of our main theorem concerning
the tree-independence number.

A \emph{tree-decomposition} of a graph~$G$ is a pair~${\mathcal{T} = (T,\beta)}$ of a tree~$T$ and a function~${\beta \colon V(T)\to 2^{V(G)}}$ whose images are called the \emph{bags} of~$\mathcal{T}$ such that 
\begin{enumerate}
    \item ${\bigcup_{u \in V(T)} \beta(u) = V(G)}$,
    \item for every~${uv \in E(G)}$ there exists some~${t \in V(T)}$ with~${\{u,v\} \subseteq \beta(t)}$, and 
    \item for every vertex~${v \in V(G)}$ the set~${\{ u \in V(T) \colon v \in \beta(u) \}}$ induces a subtree of~$T$. 
\end{enumerate}
The \emph{width} of~${(T,\beta)}$ is~${\max\{ \abs{\beta(u)}-1 \colon u \in V(T) \}}$ and the \emph{tree-width} of~$G$ is the minimum width among all possible tree-decompositions of~$G$. 

An \emph{independent set} in a graph~$G$ is a set of pairwise non-adjacent vertices. The \emph{independence number} of a graph~$G$, denoted by~${\alpha(G)}$, is the maximum size of an independent set in~$G$. 

The \emph{independence number} of~$\mathcal{T}$ is defined as  
$\alpha(\mathcal{T}) = \max_{t\in V(T)} \alpha(G[\beta(t)])$. 
    The \emph{tree-independence number} of a graph~$G$, denoted by~$\tin(G)$, is the minimum independence number among all possible tree-decompositions of~$G$. 

\begin{corollary}\label{cor:tree-independence}
For all positive integers $k$ and $t$, every $K_{1,t}$-free graph $G$
containing no induced packing of $k$ even cycles has tree-independence number $O(tk^2\log k)$.
Moreover, a tree decomposition witnessing this bound can be constructed
in polynomial time.
\end{corollary}
\begin{proof}
By Theorem~\ref{thm:main}, there is a set $X$ of size $O(k^2\log k)$
such that $G-N_G[X]$ contains no even cycle. Hence $G-N_G[X]$ is a cactus
and has tree-width at most $2$.

We take a tree-decomposition of $G-N_G[X]$ of width at most $2$ and add
$N_G[X]$ to every bag. Since $G$ is $K_{1,t}$-free, every independent
set in $N_G[X]$ has size at most $t|X|$. Thus every resulting bag has
independence number at most
$t|X|+3=O(tk^2\log k)$.
The decomposition can be obtained in polynomial time.
\end{proof}

\paragraph{Statement of AI Use}
We used ChatGPT 5.6 Sol during the preparation of this manuscript for language editing and to improve the clarity of the exposition.
All mathematical statements, proofs, and references were independently verified by the authors.

\end{document}